%% file: bwfact_v3_arxiv.tex
\include{preamble}

\newtheorem*{theorem*}{Theorem}

\usepackage{eucal}
 \DeclareMathOperator{\Fl}{{\bf Fl}}
 \newcommand{\gl}{\mathfrak{gl}}

\newcommand{\ds}{\mathrm{DS}}
\newcommand{\bos}{\mathrm{bos}}
\newcommand{\Split}{\mathrm{Split}}
\newcommand{\Fact}{\mathrm{Fact}}
\newcommand{\sdet}{\mathrm{sdet}}
\newcommand{\wt}{\widetilde}

\newcommand{\defi}[1]{{\bf \textsf{#1}}}
\newcommand{\stacks}[1]{\cite[Tag \href{https://stacks.math.columbia.edu/tag/#1}{#1}]{stacks-project}}

\title{Borel--Weil factorization for super Grassmannians}
\author{Steven V Sam}
\address{University of California, San Diego, La Jolla, CA, USA}
\email{ssam@ucsd.edu}
\urladdr{\url{https://mathweb.ucsd.edu/~ssam/}}

\date{November 7, 2025}

\begin{document}

\maketitle

\begin{abstract}
  This article deals with computing the cohomology of Schur functors applied to tautological bundles on super Grassmannians. We show that in a range of cases, the cohomology is a free module over the cohomology of the structure sheaf and that the space of generators is an irreducible representation of the general linear supergroup that can be constructed via explicit multilinear operations. Our techniques come from commutative algebra: we relate this cohomology calculation to Tor groups of certain algebraic varieties.
\end{abstract}

\tableofcontents

\section{Introduction}

\subsection{Overview}

The problem of computing sheaf cohomology of homogeneous bundles over super Grassmannians or more general homogeneous superspaces, i.e., finding a super analogue of the Borel--Weil--Bott theorem, was first posed by Manin in 1981, and only special cases are completely understood (e.g., the case of typical weights). In \cite{superres}, we added to this list of understood cases by computing the cohomology of the structure sheaf of super Grassmannians. The perspective in \cite{superres}, and the one taken in this paper, is that this problem can be approached using tools from invariant theory and commutative algebra.

The motivation behind this article is to extend the computation in \cite{superres} to more general homogeneous bundles. The super Grassmannian $X =\Gr(p|q,\bC^{m|n})$ has a tautological exact sequence of vector bundles
\[
  0 \to \cR \to \bC^{m|n} \times X \to \cQ \to 0.
\]
In the classical (non-super) setting (i.e., when $n=0$), every irreducible homogeneous bundle (up to a twist by a line bundle) can be constructed using Schur functors and is of the form $\bS_\alpha \cQ \otimes \bS_\beta(\cR^*)$. The Borel--Weil theorem implies that this bundle has no higher cohomology and that the space of sections is an irreducible representation of $\GL(m)$ that can be explicitly determined. For our purposes, this representation can be constructed as a rational Schur functor. (The full Borel--Weil--Bott theorem allows us to incorporate arbitrary twists, but this article will not deal with this generalization.) The main theorem of this paper generalizes this calculation (we give here a preliminary version, see Theorem~\ref{thm:main-Q} for the full statement, see Theorem~\ref{thm:relative} for a relative version, and see Theorem~\ref{thm:partialflag} for a generalization to super partial flag varieties):

\begin{thm}[preliminary version]
  Assume that $m-n -\ell(\alpha) \ge p-q \ge \ell(\beta)$ and set $\cE = \bS_\alpha(\cQ) \otimes \bS_\beta(\cR^*)$. Then $\rH^\bullet(X; \cE)$ is a free $\rH^\bullet(X; \cO_X)$-module and we have a graded $\GL(m|n)$-equivariant isomorphism
  \[
    \rH^\bullet(X; \cE) \cong \rH^\bullet(X; \cO_X)\otimes \rH^0(X; \cE).
  \]
  Furthermore, $\rH^0(X; \cE)$ is an irreducible $\GL(m|n)$-representation and can be constructed as a rational Schur functor.
\end{thm}

While not true in general, under our assumptions on $p,q,m,n$, the $\GL(m|n)$-action on $\rH^\bullet(X;\cO_X)$ is trivial.

The rational Schur functor is an explicit construction that we will detail in \S\ref{sec:rational-schur}, which can be read independently from the rest of the article. We point out that irreducibility does not always hold, unlike the non-super case. The proof that they are irreducible in our cases of interest might be of independent interest. Furthermore, their characters are instances of composite supersymmetric Schur polynomials.

We think of this isomorphism as being a ``factorization'' of the cohomology of $\cE$ which mixes the features from the classical Borel--Weil theorem with genuinely new features in the super setting (i.e., the fact that $\rH^\bullet(X;\cO_X)$ is a non-trivial algebra). We note that in the non-super setting (when $n=0$, and hence $q=0$), the assumptions can be written as $m-p \ge \ell(\alpha)$ and $p \ge \ell(\beta)$, which are equivalent to the sheaf $\cE$ being nonzero, so the conditions are superfluous. However, when $n>0$, these conditions are genuine restrictions, and when they are loosened, various conclusions can fail (see \S\ref{sec:examples}).

We note that a similar theorem holds for orthosymplectic Grassmannians, but this will be handled in a subsequent article \cite{bwfact2}.

~

The rest of the introduction is devoted to explaining the main ideas that go into the proof of our result. We will first explain the main idea from \cite{superres} since our ideas build on it in a few different ways.

The main idea in \cite{superres} used to compute $\rH^\bullet(X;\cO_X)$ involves degenerations. Namely, let $X$ be a  noetherian superscheme and let $\cJ \subseteq \cO_X$ denote the ideal sheaf generated by all odd degree elements in its structure sheaf $\cO_X$. For every finitely generated $\cO_X$-module $\cM$, we consider the $\cJ$-adic filtration, which gives a convergent spectral sequence starting from the cohomology of its associated graded sheaf $\gr(\cM) = \bigoplus_{i \ge 0} \cJ^i\cM/\cJ^{i+1}\cM$ and ending with the cohomology of $\cM$. The advantage of this is that $\gr(\cM)$ is a sheaf over a scheme $X_\bos$, the bosonic reduction of $X$, which has the same underlying topological space as $X$ but now with structure sheaf $\cO_X/\cJ$.

When $X$ is a super Grassmannian $\Gr(p|q,\bC^{m|n})$, its bosonic reduction is a product of two Grassmannians $\Gr(p,\bC^m) \times \Gr(q,\bC^n)$, and one can at least attempt to compute the input to this spectral sequence using classical tools, such as the Borel--Weil--Bott theorem. The combinatorics becomes cumbersome to keep track of quite quickly, but we can use a different approach. The associated graded sheaf of $\cO_X$ is an exterior algebra $\bigwedge^\bullet(\cJ/\cJ^2)$, and in commutative algebraic contexts, it is natural for exterior algebras to appear as Koszul complexes.

This interpretation can be exploited: $\cJ/\cJ^2$ is naturally a subbundle of the trivial bundle $\fgl(m|n)^*_1 \times X_\bos$, where $\fgl(m|n)_1$ is the odd component of the Lie superalgebra $\fgl(m|n)$, so we can interpret $\cJ/\cJ^2$ as (locally linear) equations cutting out a subbundle $Y_0$ in $\fgl(m|n)_1 \times X_\bos$ (in \cite{superres}, the space is called $Y$, but we will use $Y$ for a variant of this construction in this article). So we have a projection map $\pi \colon Y_0 \to \fgl(m|n)_1$. (This is an example of a Kempf collapsing, or from another perspective, a super analogue of the Grothendieck--Springer resolution.)

Homologically speaking, we can equate the Koszul complex $\bigwedge^\bullet(\cJ/\cJ^2)$ with the structure sheaf $\cO_{Y_0}$, so that the artifacts of the super structure are recast as something in the realm of non-super algebraic geometry. This heuristic can be made concrete: by cohomology and base change, the cohomology groups of $\bigwedge^\bullet(\cJ/\cJ^2)$ compute the Tor groups of the (derived) pushforward of $\cO_{Y_0}$ along $\pi$. Part of the main result of \cite{superres} is that the higher direct images in fact vanish, and the pushforward of $\cO_{Y_0}$ is a finite rank free module over a determinantal variety. The Tor groups of the latter are known, and the exact form of the answer implies that the original spectral sequence degenerates.

Now replace $\cO_X$ by the homogeneous bundle $\cE = \bS_\alpha(\cQ) \otimes\bS_\beta(\cR^*)$ from the main theorem. If we proceed as before, then $\gr(\cE) = (\cE/\cJ\cE) \otimes \bigwedge^\bullet(\cJ/\cJ^2)$ and we end up with a coherent sheaf $\cE/\cJ\cE$ over the space $Y_0$. Cohomology and base change again implies that the relevant cohomology groups compute the Tor groups of the derived pushforward of this coherent sheaf. It is not clear what the meaning of this coherent sheaf is, so we take a different perspective. We introduce auxiliary vector spaces $E$ and $F$ with $\dim E = (m-n)-(p-q)$ and $\dim F = p-q$. The Cauchy identity gives a natural decomposition
\[
  \Sym( E \otimes \cQ) \otimes \Sym (\cR^* \otimes F) = \bigoplus_{\alpha, \beta} \bS_\alpha E \otimes \bS_\alpha \cQ \otimes \bS_\beta (\cR^*) \otimes \bS_\beta F
\]
where the sum is over all partitions $\alpha, \beta$ such that $\ell(\alpha) \le \dim E$ and $\ell(\beta) \le \dim F$, i.e., exactly the pairs of partitions that satisfy the assumptions in the main theorem. If we keep track of the $\GL(E) \times \GL(F)$ action, then this allows us to package all of the computations together. The advantage of this is that the left side has an algebra structure. Translating to non-super notation, the left side naturally decomposes into symmetric and exterior algebras:
\begin{align*}
  \Sym(E \otimes \cQ) &= \Sym(E \otimes \cQ_0) \otimes \bigwedge^\bullet(E \otimes \cQ_1),\\
  \Sym(\cR^* \otimes F) &= \Sym(\cR^*_0 \otimes F) \otimes \bigwedge^\bullet(\cR^*_1 \otimes F).
\end{align*}

We wish to interpret this algebra in terms of the space $Y_0$ from before. A symmetric algebra has a natural interpretation as the coordinate ring of a vector space (or vector bundle), so the presence of $\Sym(E \otimes \cQ_0)$ and $\Sym(\cR_0^* \otimes F)$ suggest to work with the total space of the vector bundle $(E^* \otimes \cQ_0^*) \oplus (\cR_0 \otimes F^*)$ over $Y_0$. To incorporate the exterior algebras, we use the heuristic from before, and try to interpret them as Koszul complexes, i.e., equations defining a complete intersection. Unlike before, there is no natural way to interpret them as linear equations, but we can interpret them as quadratic equations (we omit this here due to its technical nature, a more detailed overview can be found in \S\ref{ss:outline}). This leads to a complication: linear equations (in a vector bundle) always generate a complete intersection, but for quadratic equations, this is a non-trivial condition. This constraint is responsible for the inequalities in the assumption of the main theorem.

The technicalities aside, the above construction gives us a space $Y$ over $Y_0$ and again we can consider the projection map $Y \to \fgl(m|n)_1$. The pushforward of $\cO_Y$ is no longer a finitely generated module, but we will show that a similar relation holds: the cohomology groups we are interested in compute the Tor groups of the derived pushforward of $\cO_Y$; the latter can also be explicitly computed and again the relevant spectral sequence degenerates, which will allow us to prove the main theorem.

\subsection{Outline of paper}

There are a lot of moving parts to the technical aspects glossed over above; we end the introduction with an outline of the paper.

Section~\ref{sec:bg} contains background information and does not contain any new results. We review some important concepts from super linear algebra (in particular, Schur functors) and super geometry. The last part on splitting rings is taken from \cite{superres}. As mentioned above, we will obtain some finite rank covers of (more familiar) varieties through the course of our calculations. The splitting ring construction is important for identifying what these covers are. More importantly, we will use this information to prove that these covers have rational singularities, which is highly relevant for computing the derived pushforward mentioned above.

Section~\ref{sec:CI} is devoted to studying a class of complete intersections which are variations of the variety of complexes. The variety of complexes appears as the quadratic complete intersection mentioned above in the construction of $Y$. But to understand the pushforward of $\cO_Y$, we will need to consider the more general class. We will also prove that the spaces have rational singularities. To the best of our knowledge, this construction has not been considered before in the literature.

Section~\ref{sec:rational-schur} focuses on the construction of rational Schur functors for super vector spaces. Roughly speaking, this is formed by taking a certain quotient of a tensor product $\bS_\lambda(V) \otimes \bS_\mu(V^*)$. In the non-super setting, these give non-zero irreducible representations when $\dim V \ge \ell(\lambda) + \ell(\mu)$ and, in fact, exhaust the set of all irreducible rational representations of $\GL(V)$. They can be constructed using some invariant theory related to the variety of complexes, and we show that a similar approach can be done in the super setting. We also identify their characters with the composite supersymmetric Schur polynomials.

Under the assumption $\dim V_0 - \dim V_1 \ge \ell(\lambda) + \ell(\mu)$, we show that they give irreducible representations of the general linear supergroup $\GL(V)$. This is more difficult than the non-super case for two reasons: (1) general finite-dimensional representations of $\GL(V)$ need not be semisimple, and (2) rational Schur functors can be reducible in general, so this inequality must be used in some nontrivial way.

Section~\ref{sec:main-thm} uses the previous sections to state the main theorem and prove it. A vast majority of the work is to prove that certain subvarieties are complete intersections and have rational singularities; this relies on the material in Sections~\ref{sec:bg} and~\ref{sec:CI}. The degeneration approach outlined above allows us to identify the cohomology $\rH^\bullet(X; \cE)$ as a $\GL(V_0) \times \GL(V_1)$-representation, and the material in Section~\ref{sec:rational-schur} allows us to upgrade this to an identification as a $\GL(V)$-representation.

Section~\ref{sec:misc} contains some further results and generalizations. We generalize the main theorem to the setting of relative super Grassmannians. Using this, we also deduce a generalization to super partial flag varieties using the Leray spectral sequence.

Finally, in Section~\ref{sec:examples}, we give some examples to illustrate what can go wrong if the assumptions of the main theorem do not hold. We also make some remarks about possible future extensions of this work.

\subsection{Relation to other work}

Of course, the main inspiration for this article is \cite{superres}. In the non-super case, the general idea of relating cohomology of homogeneous bundles to Tor groups goes back to Lascoux \cite{lascoux} and an expository treatment of this idea and many important examples appear in Weyman's book \cite{weyman}.

The idea of connecting Tor groups of determinantal varieties with representations of general linear Lie superalgebras goes back to work of Pragacz and Weyman \cite{pragacz-weyman} and later Akin and Weyman \cite{akin-weyman, akin-weyman2, akin-weyman3}, and subsequent work can be found in \cite{raicu-weyman, raicu-weyman2, huang}.

A variation with periplectic Lie algebras is studied in \cite{sam} and the connection to cohomology of periplectic Grassmannians will appear in \cite{superres2}. A different variation which involves infinite-dimensional representations of orthosymplectic Lie superalgebras appears in \cite{sam-osp}. The connection of this work to cohomology of orthosymplectic Grassmannians is unclear; however, an analogue of the main theorem of this article for these spaces will appear in \cite{bwfact2}.

For some other papers studying the cohomology of homogeneous bundles on homogeneous superspaces, see \cite{coulembier, gruson, penkov, penkov-serganova}.

\medskip

\noindent {\bf Acknowledgements.} We thank Abhik Pal and Andrew Snowden for helpful discussions, and an anonymous referees for their comments. The author was partially supported by NSF grant DMS-2302149.

\section{Background and notation} \label{sec:bg}

\subsection{Partitions and super linear algebra} \label{ss:linear-algebra}

If $A$ and $B$ are vector spaces (respectively, vector bundles), we use the notation $A|B$ to denote the superspace (respectively, super bundle) with even part $A$ and odd part $B$. To be clear: the notation $B$ means we treat $B$ as a vector space or bundle with no $\bZ/2$-grading, and $0|B$ means that we treat it as a space or bundle concentrated in odd degree. We will work over a field of characteristic 0. The dimension of a super vector space $E$ will be denoted $\dim E_0|\dim E_1$; we use similar notation for ranks of super vector bundles.

Given two superspaces $E,F$, we let $\hom(E,F)$ denote the space of linear maps $E \to F$. This inherits the structure of a superspace using the grading:
\begin{align*}
  \hom(E,F)_0 &= \hom(E_0,F_0) \oplus \hom(E_1,F_1),\\
  \hom(E,F)_1 &= \hom(E_0,F_1) \oplus \hom(E_1,F_0).
\end{align*}

We will assume familiarity with integer partitions, see for example \cite[\S 1.1.2]{weyman}, but state some definitions to establish notation. Partitions are weakly decreasing sequences of non-negative integers and will be denoted by Greek letters, e.g., $\lambda =(\lambda_1,\dots,\lambda_r)$. The size is $|\lambda|=\sum_i \lambda_i$, and the length is $\ell(\lambda)$, which is the number of nonzero entries in $\lambda$. The transpose partition is denoted $\lambda^T$, which is defined by $\lambda^T_i = \# \{j \mid \lambda_j \ge i\}$.

We use two partial orders: the inclusion order $\mu \subseteq \lambda$ means that $\mu_i \le \lambda_i$ for all $i$, and the \defi{dominance order} $\mu \le \lambda$ means $\sum_{i=1}^r \mu_i \le \sum_{i=1}^r \lambda_i$ for all $r$. The latter will only be used to compare partitions of the same size. If $\mu \subseteq \lambda$, we use $\lambda/\mu$ to denote a skew partition.

We will extensively use Schur functors applied to super vector spaces (bundles), see \cite[\S 2.4]{weyman} for a detailed exposition. These are defined for any skew partition $\lambda/\mu$. One notable difference: our notation $\bS_{\lambda/\mu}$ for Schur functors coincides with $L_{{\lambda^T/\mu^T}}$ from \cite{weyman}. The Littlewood--Richardson coefficients are denoted $c_{\lambda,\mu}^\nu$, and are defined via decomposition of tensor products
\[
  \bS_\lambda \otimes \bS_\mu \cong \bigoplus_\nu \bS_\nu^{\oplus c^{\nu}_{\lambda, \mu}},
\]
see \cite[\S 2.3]{weyman} for more details. If $c^\nu_{\lambda,\mu} \ne 0$, then:
\begin{itemize}
\item $|\lambda|+|\mu|=|\nu|$,
\item $\lambda \subseteq\nu$ and $\mu\subseteq \nu$, and
\item $\lambda+\mu \ge \nu$.
\end{itemize}
The first two properties are well-known and follow from the many well-known combinatorial rules for computing these coefficients. The third property is likely also well-known, but see \cite[Proposition 3.2.1]{quot} for a proof.

If $V$ is any super vector bundle, then we have
\begin{align*} \label{eqn:schur-super}
  \bS_\lambda(V_0|V_1) \cong \bigoplus_{\mu \subseteq \lambda} \bS_\mu(V_0) \otimes \bS_{\lambda^T/\mu^T}(V_1),
\end{align*}
see \cite[Theorem 2.4.6]{weyman}; this is stated for modules only, but the isomorphisms are compatible with change of bases of $V_0$ and $V_1$, so they extend to vector bundles. Symmetric powers $\Sym^d$ are the special case $\lambda=(d)$, i.e., $\Sym^d = \bS_{(d)}$. Similarly, exterior powers $\bigwedge^d$ agree with $\bS_{(1,\dots,1)}$ ($d$ 1's). 

  If $V,W$ are super vector bundles, the \defi{Cauchy identity} gives decompositions (functorial in $V$ and $W$)
\[
  \Sym(V \otimes W) \cong \bigoplus_{\lambda} \bS_\lambda V \otimes \bS_\lambda W, \qquad
  \bigwedge^\bullet(V \otimes W) \cong \bigoplus_{\lambda} \bS_\lambda V \otimes \bS_{\lambda^T}W,
\]
where the sum is over all partitions $\lambda$. See \cite[Theorem 2.3.2]{weyman}; this is only stated for non-super vector spaces, but since it is independent of the dimensions of these spaces, it is actually an isomorphism of functors, so they are valid for any input, see \cite[(5.4.3)]{expos} for this perspective.

\subsection{Superschemes} \label{ss:supergeom}

For details on super algebraic geometry, see \cite{FAS, manin} and for some more specific information on the super Grassmannian which is relevant for our purposes, see \cite[\S 6.2]{superres}. Below, ``superalgebra'' will mean a $\bZ/2$-graded algebra which is graded-commutative. Just as in (non-super) algebraic geometry, superschemes can be described in two different ways: either as representable functors from the category of superalgebras to sets, or as spaces with a sheaf of superalgebras. We will use the former perspective to define our spaces of interest, but the latter perspective will be important for our main application.

Throughout, we will always work over the complex numbers $\bC$. (Everything here can also be done over an arbitrary algebraically closed field of characteristic 0 without any substantial change; this choice is largely stylistic.) In this article, we will be concerned with homogeneous spaces for the general linear supergroup $\GL(m|n)$. Let $V = V_0|V_1$ be a super vector space of dimension $m|n$.

First, $\GL(V)=\GL(m|n)$ represents the functor that assigns to a superalgebra $T$ the group of invertible linear operators of the free $T$-module $T^{m|n} = T \otimes V$. If $p|q$ is a pair of non-negative numbers, we write $p|q < p'|q'$ to mean that $p \le p'$ and $q \le q'$ and that at least one of those inequalities is strict. Given a sequence $p_1|q_1 < p_2|q_2 < \cdots < p_r|q_r < m|n$, we have a functor that assigns to each superalgebra $T$ the set of sequences of $T$-submodules $R_1 \subseteq R_2 \subseteq \cdots \subseteq R_r \subseteq T^{m|n}$ such that each $R_i$ is locally a summand of $T^{m|n}$ of rank $p_i|q_i$. This functor is representable by a superscheme which we will call the \defi{super partial flag variety} and denote by ${\bf Fl}(\bp|\bq, V)$ or ${\bf Fl}(\bp|\bq, m|n)$. Of particular interest for us will be the case when $r=1$, in which case we get the \defi{super Grassmannian}, which we denote by $\Gr(p|q,V)$ (where $p=p_1$ and $q=q_1$).

Each superscheme $X$ can also be thought as a space equipped with a sheaf of superalgebras $\cO_X$. Let $\cJ \subset \cO_X$ be the ideal sheaf generated by all odd degree elements. The quotient $\cO_X/\cJ$ is a sheaf of commutative algebras, and following \cite{FAS}, we denote the corresponding ringed space (in fact, it is a scheme) as $X_\bos$, the \defi{bosonic reduction} of $X$.

Given an $\cO_X$-module $\cM$, it inherits a $\cJ$-adic filtration: $\cM \supset \cJ \cM \supset \cJ^2 \cM \supset \cdots$ and we let
\[
  \gr \cM = \bigoplus_{i \ge 0} \cJ^i M / \cJ^{i+1} M
\]
denote the associated graded sheaf with respect to the $\cJ$-adic filtration. This filtration gives a spectral sequence on cohomology:
\[
  \rE^{p,q}_1 = \rH^p(X_\bos; \gr^q \cM) \Longrightarrow \rH^p(X; \cM),
\]
and we say that $\cM$ is \defi{$\cJ$-formal} if this spectral sequence is convergent (in cases of interest, the $\cJ$-adic filtration is finite, so this is automatic) and degenerates on the first page; in particular, if $\cM$ is $\cJ$-formal, then for all $p$ have an isomorphism of vector spaces
\[
  \rH^p(X; \cM) \cong \rH^p(X_\bos; \gr \cM).
\]

\begin{remark}
  A consequence of the main theorem of \cite{superres} is that the structure sheaf of a super Grassmannian is $\cJ$-formal. But this property can fail for general vector bundles on super Grassmannians, see Example~\ref{ex:ex2} for one such example.
\end{remark}

We get a surjective map
\[
  \bigwedge^\bullet(\cJ/\cJ^2) \to \gr \cO_X,
\]
and a superscheme $X$ is called smooth if $X_\bos$ is smooth, $\cJ/\cJ^2$ is a locally free sheaf on $X_\bos$, and this map is an isomorphism. If $\cE$ is a locally free sheaf over a smooth superscheme $X$, then
\[
  \gr \cE \cong (\cE/\cJ \cE) \otimes \bigwedge^\bullet(\cJ/\cJ^2)
\]
(see \cite[\S 3, Proposition 14]{manin}). We will write $\gr^0 \cE = \cE/\cJ \cE$. By standard properties of tensor products, given locally free sheaves $\cE$ and $\cF$ over $X$, we have
\[
  \gr^0 (\cE \otimes_{\cO_X} \cF) \cong \gr^0 \cE \otimes_{\cO_{X_\bos}} \gr^0 \cF.
\]
Since the construction of Schur functors are compatible with base change (this follows from their definition as the image of a composition of maps built out of exterior power comultiplication and symmetric power multiplication, as in \cite[\S 2.4]{weyman}), we also have
\[
  \gr^0(\bS_\lambda(\cE)) \cong \bS_\lambda(\gr^0\cE).
\]

The bosonic reduction of $X = {\bf Fl}(\bp|\bq, m|n)$ is the product of two partial flag varieties ${\bf Fl}(\bp,m) \times {\bf Fl}(\bq,n)$. The sequences $\bp$ and $\bq$ may have repetitions, but for convenience, we will remember this information.

We can express a super partial flag variety $X$ as a coset space $\GL(V)/\bP$ where $\bP$ is the stabilizer of a particular flag (see \cite{masuoka, MZ} for details on quotient spaces of supergroups). Hence the super partial flag varieties are all smooth \cite[Proposition 4.16]{masuoka}. Let $\fg = \fgl(V)$ be the Lie superalgebra of $\GL(V)$ and let $\fp$ be the Lie superalgebra of $\bP$. Then $\cJ/\cJ^2$ is the homogeneous bundle on $X_\bos$ associated to the $\bP_\bos$-module $(\fg_1/\fp_1)^*$ (via the standard induction functor from $\bP_\bos$-modules to homogeneous bundles over $\GL(V)_\bos/\bP_\bos$) \cite[Proposition 4.18]{masuoka}.

In particular, we can think of $\cJ/\cJ^2$ as (locally) linear equations inside of the trivial bundle $\fg_1 \times X_\bos$. The zero locus of these equations is the total space of the homogeneous bundle on $X_\bos$ associated to $\fp_1^*$; denote it by $Y$. The space $Y$ can be described as the space of tuples $(R_{0,\bullet}, R_{1,\bullet}, f, g)$ where
\begin{itemize}
\item $R_{0,\bullet} \in {\bf Fl}(\bp,V_0)$, 
\item $R_{1,\bullet} \in {\bf Fl}(\bq, V_1)$, 
\item $f \colon V_0 \to V_1$ is a linear map such that $f(R_{0,i}) \subseteq R_{1,i}$ for all $i$, and  
\item $g \colon V_1 \to V_0$ is a linear map such that $f(R_{1,i}) \subseteq R_{0,i}$ for all $i$.
\end{itemize}

Now consider the case of the super Grassmannian $\Gr(p|q,V)$. There is a tautological exact sequence
\[
  0 \to \cR \to V \times \Gr(p|q,V) \to \cQ \to 0.
\]
Then $X_\bos = \Gr(p,V_0) \times \Gr(q,V_1)$. Let $\cR_0$ denote the rank $p$ tautological subbundle of $V_0 \times \Gr(p,V_0)$, and let $\cQ_0$ be its quotient. Define $\cR_1$ and $\cQ_1$ similarly. We extend this notation to also include their pullbacks to $X_\bos$. Then in this case, we have
\[
  \cJ / \cJ^2 \cong (\cR_0 \otimes \cQ_1^*) \oplus (\cQ_0^* \otimes \cR_1)
\]
and $\gr^0 \cR = \cR_0 \oplus \cR_1$ and $\gr^0 \cQ= \cQ_0 \oplus \cQ_1$ and similarly for their duals. Putting everything together, for any two partitions $\alpha,\beta$ we have
\[
  \gr( \bS_\alpha \cQ \otimes \bS_\beta(\cR^*)) = \bS_\alpha(\cQ_0|\cQ_1) \otimes \bS_\beta(\cR^*_0|\cR^*_1) \otimes \bigwedge^\bullet((\cR_0 \otimes \cQ_1^*) \oplus (\cQ_0^* \otimes \cR_1)),
\]
which will be used later in this paper.

\subsection{Splitting rings and factorization rings} \label{ss:fact}

In this section, we will review the constructions and results from \cite[\S 3]{superres} that are relevant for this paper.

Let $A$ be a ring and let $f(u) = u^n + \sum_{i=0}^{n-1} a_{n-i} u^i$ be a monic polynomial with coefficients in $A$. The \defi{splitting ring} of $f$, denoted $\Split_A(f)$, is the quotient of $A[\xi_1,\dots,\xi_n]$ obtained by identifying the coefficients of the polynomials:
\[
  f(u) = (u-\xi_1) \cdots (u-\xi_n).
\]
Then $\Split_A(f)$ is a free $A$-module of rank $n!$ \cite[Proposition 3.1]{superres}. There is a natural $\fS_n$-action on $\Split_A(f)$ by permuting the $\xi_i$.

If $A$ is graded and $\deg(a_i) = i \deg (a_1)$ for all $i$, then $\Split_A(f)$ inherits a grading by setting $\deg(\xi_i) = \deg(a_1)$ for all $i$. From the definition, it follows that if $\bk$ is an algebraically closed field, then a homomorphism $\Split_A(f) \to \bk$ consists of a homomorphism $\phi \colon A \to \bk$ together with an ordering of the roots of $\phi(f)$ (the polynomial obtained from applying $\phi$ to the coefficients of $f$).

We will need to know when $\Split_A(f)$ is normal, and for that, we recall \cite[Proposition 3.10]{superres}. Let $\Delta \in A$ be the discriminant of $f$. We define
\[
  V(\Delta, \partial \Delta) = \{ \fp\in \Spec A \mid \Delta \in \fp^2 A_\fp\}.
\]

\begin{proposition} \label{prop:split-normal}
  $\Split_A(f)$ is normal if the following $3$ conditions hold:
  \begin{enumerate}
  \item $A$ is normal,
  \item $\Delta$ is a nonzerodivisor, and
  \item $V(\Delta, \partial \Delta)$ has codimension $\ge 2$ in $\Spec A$.
  \end{enumerate}
\end{proposition}

Now let $\bd = (d_1,d_2,\dots,d_r)$ be a sequence of non-negative integers whose sum is $n$ (we allow $d_i=0$ for convenience) and again let $f = u^n + \sum_{i=0}^{n-1} a_{n-i} u^i$ be a monic polynomial of degree $n$. Consider a polynomial ring $A[\bb] = A[b_{i,j} \mid 1 \le i \le r, \ 1 \le j \le d_i ]$ and for $i=1,\dots,r$, set $g_i = u^{d_i} + \sum_{j=0}^{d_i-1} b_{i,d_i-j} u^j$ (in the case $d_i=0$, we simply have $g_i = 1$). The \defi{$\bd$-factorization ring} of $f$, denoted $\Fact^\bd_A(f)$, is the quotient of $A[\bb]$ obtained by identifying the coefficients of the polynomials
\[
  f(u) = g_1(u)\cdots g_r(u).
\]
As with splitting rings, if $A$ is graded with $\deg(a_i)=i\deg(a_1)$ for all $i$, then $\Fact^\bd_A(f)$ inherits a grading by setting $\deg(b_{i,j}) = j \deg(a_1)$ for all $i,j$.

This construction is discussed in \cite[\S 3.7]{superres} for $r=2$. Properties for the general case follow from this case, so we will just state the relevant generalizations and give references for the proofs.

\begin{proposition} \label{prop:fact-ring}
  \begin{enumerate}
  \item Let $C = \Fact^{\bd}_A(f)$. We have a natural isomorphism of $A$-modules
    \[
      \Split_A(f) = \Split_C(g_1) \otimes_C \cdots \otimes_C \Split_C(g_r).
    \]
  \item $\Fact^\bd_A(f)$ is a free $A$-module of rank $\DS \frac{n!}{d_1! \cdots d_r!}$.
  \item If $n!$ is invertible, then $\Fact^\bd_A(f) = \Split_A(f)^G$ where $G \cong \fS_{d_1} \times \cdots \times \fS_{d_r}$ is the subgroup of $\fS_n$ where $\fS_{d_1}$ permutes $\xi_1,\dots,\xi_{d_1}$, $\fS_{d_2}$ permutes $\xi_{d_1+1},\dots, \xi_{d_1+d_2}$, etc, and the superscript denotes taking $G$-invariants. 
  \end{enumerate}
\end{proposition}

\begin{proof}
  (1) The proof of \cite[Proposition 3.11(a)]{superres} can be modified to prove this statement.

  (2) This follows from (1) and the fact that each $\Split_C(g_i)$ is a free $C$-module of rank $d_i!$.

  (3) If $n!$ is invertible, then each $\Split_C(g_i)$ is a free $C[\fS_{d_i}]$-module of rank 1 \cite[Proposition 3.1(a)]{superres}, so the result follows from (1).
\end{proof}

In the case when $A$ is the Chow ring of a smooth variety $X$, the Chow ring of any partial flag bundle over $X$ can be realized as $\Fact^\bd_A(f)$ for appropriate choices of $\bd$ and $f$ (see \cite[Theorem 5.1]{GSS} for details). When $X$ is a point, this gives us an interpretation for $\Fact_A^\bd(u^n)$. To state the result, let $\wt{d}_i = d_1 + d_2 + \cdots + d_i$, and let $\Fl(\wt{\bd},\bC^n)$ denote the partial flag variety whose points parametrize sequences of subspaces $R_1 \subset \cdots \subset R_r \subset \bC^n$ where $\dim R_i = \wt{d}_i$. Below, we will use singular cohomology rather than the Chow ring since it is likely a more familiar object; these agree for $\Fl(\wt{\bd},\bC^n)$ since it has a cellular decomposition.

\begin{proposition} \label{prop:fact-flag}
  View $A$ as a graded ring which is concentrated in degree $0$. Then we have a natural isomorphism of graded rings
  \[
    \Fact^\bd_A(u^n) = A \otimes \rH^\bullet_{\rm sing}(\Fl(\wt{\bd},\bC^n); \bZ)
  \]
  where the generators $b_{i,j}$ of $\Fact^\bd_A(u^n)$ have degrees given by $\deg b_{i,j} = 2j$.
\end{proposition}

When we take $A$ to be a field, for example $\bC$ as in this paper, the Hilbert series can be given in a compact form. Define $[i] = \frac{t^{2i}-1}{t^2-1} = 1 + t^2 + \cdots + t^{2i-2}$ and $[i]! = [i][i-1]\cdots[1]$. Then
\[
  \sum_i \dim_\bC \rH^i_{\rm sing}(\Fl(\wt{\bd},\bC^n); \bC) t^i = \frac{[n]!}{[d_1]! \cdots [d_r]!}.
\]
This fact is well-known, but can also be deduced from the presentation of the factorization ring given above as a complete intersection. In the special case $r=2$, we get the Grassmannian, and the formula becomes
\begin{align} \label{eqn:gr-sing-hilbert}
  \sum_i \dim_\bC \rH^i_{\rm sing}(\Gr(q,\bC^n); \bC) t^i =   \frac{\prod_{i=q+1}^n (t^{2i}-1)}{\prod_{i=1}^{n-q} (t^{2i}-1)}.
\end{align}

\section{Some complete intersections} \label{sec:CI}

Let $X$ be a finite type scheme over an algebraically closed field of characteristic 0 (but see Remark~\ref{rmk:pos-char}), and let $\cA_1$, $\cA_2$, $\cB$, $\cC_1$, $\cC_2$ be vector bundles over $X$ of ranks $a_1,a_2,b,c_1,c_2$, respectively. Let $\bX$ be the total space of the vector bundle
\[
  \hom(\cA_1, \cB) \oplus \hom(\cB, \cA_2) \oplus \hom(\cC_1, \cB) \oplus \hom(\cB, \cC_2).
\]
When we pull these bundles back to $\bX$, there are tautological maps
\[
  \alpha_1 \colon \cA_1 \to \cB, \qquad \alpha_2 \colon \cB \to \cA_2, \qquad \gamma_1 \colon \cC_1 \to \cB, \qquad \gamma_2 \colon \cB \to \cC_2.
\]
Let $Z$ denote the common scheme-theoretic zero locus of the three compositions:
\begin{itemize}
\item $\alpha_2\alpha_1 \colon \cA_1 \to \cA_2$,
\item $\gamma_2 \alpha_1 \colon \cA_1 \to \cC_2$,
\item $\alpha_2 \gamma_1 \colon \cC_1 \to \cA_2$.
\end{itemize}

\begin{proposition} \label{prop:complex-CI}
  Assume that $a_1+a_2+\max(c_1,c_2) \le b$. The map $Z \to X$ is flat, and all of the fibers are reduced, irreducible, locally complete intersections (specifically, the equations for $Z$ are regular), and have rational singularities. In particular, if $X$ is reduced,  respectively, irreducible, locally a complete intersection, normal, or has rational singularities, then the same holds for $Z$.
\end{proposition}

\begin{proof}
  Flatness of $Z\to X$ will follow once we show that the equations defining $Z$ are regular since then the fibers are finitely generated graded algebras with constant Hilbert function (and so $\cO_Z$ is locally free over $\cO_X$). So it suffices to handle the case that $X$ is a point and $A_1, A_2, B, C_1, C_2$ are all vector spaces since all of the properties we listed are local and closed under taking products. In what follows, we let $\phi_1 \colon A_1 \oplus C_1 \to B$ denote the sum $\alpha_1 + \gamma_1$ and let $\phi_2 \colon B \to A_2 \oplus C_2$ denote the sum $\alpha_2 \oplus \gamma_2$.

Our proof will use some interactions between the stated properties which rely on one another, so here is an outline before we begin:
\begin{enumerate}
\item We construct a desingularization $\pi \colon \wt{Z} \to Z_{\rm red}$ of the reduced scheme of $Z$, by a vector bundle over a partial flag variety. This implies that $Z$ is irreducible and lets us show that $\codim(Z,X) = a_1a_2+a_1c_2+c_1a_2$. Since this is the number of equations defining $Z$, we see that it is a complete intersection, and in particular, Cohen--Macaulay.
  
\item Next, we construct a closed subvariety $Z^0$ such that the Jacobian matrix of these equations has maximal possible rank $\codim(Z,X)$ at points of $U = Z \setminus Z^0$.
  
\item We directly show that $\codim(\pi^{-1}(Z^0),\wt{Z}) \ge 2$. This implies that $\codim(Z^0,Z) \ge 2$ and then by Serre's criterion, we conclude that $Z$ is reduced \stacks{031R} and normal \stacks{031S}. This will also imply that $Z$ has rational singularities by \cite[Proposition 4.1]{superres}.
\end{enumerate}
  
 {\bf Step 1:}  We construct a desingularization of $Z_{\rm red}$ as follows: first consider the partial flag variety $\bF = {\bf Fl}(a_1, a_1+c_1, B)$ whose points parametrize subspaces $R \subset S \subset B$ where $\rank R = a_1$ and $\rank S = a_1+c_1$. Let $\cR \subset \cS$ be the tautological subbundles of the trivial bundle $B$ over $\bF$. Let $\eta$ be the vector bundle
  \[
   \eta= \hom(A_1, \cR)^* \oplus \hom(C_1, \cS)^* \oplus \hom(B/\cS, A_2)^* \oplus \hom(B/\cR, C_2)^*,
  \]
  and let $\wt{Z} = \Spec_{\bF}(\Sym(\eta))$. We claim that the image of the projection map $\pi \colon \wt{Z} \to \bX$ is $Z_{\rm red}$. It is clear that $\pi(\wt{Z}) \subseteq Z_{\rm red}$, so we will show the other containment. Given $(\alpha_1,\alpha_2, \gamma_1, \gamma_2) \in Z_{\rm red}$, we have a containment $\alpha_1(A_1) \subseteq \ker \phi_2$. Let $R$ be any $a_1$-dimensional subspace that sits between these spaces. This is possible since
  \[
    \dim \alpha_1(A_1) \le a_1 \le b-a_2-c_2 \le \dim \ker \phi_2.
  \]
  Similarly, we have $R + \phi_1(A_1 \oplus C_1) \subseteq \ker \alpha_2$ and we choose $S$ to be any $(a_1+c_1)$-dimensional subspace that sits in between these spaces. Again, this is possible since
  \[
    \dim (R + \phi_1(A_1 \oplus C_1)) \le a_1 + c_1 \le b - a_2 \le \dim \ker \alpha_2.
  \]
  By construction, $\alpha_1(A_1) \subseteq R \subseteq \ker \phi_2$ and $\gamma_1(C_1) \subseteq S \subseteq \ker \alpha_2$, and so this point is in the image of $\pi$. In particular, $Z_{\rm red} = \pi(\wt{Z})$, and we see that $Z$ is irreducible.

  The restriction $\pi \colon \wt{Z} \to Z_{\rm red}$ is birational: there is a nonempty open subset of $Z$ consisting of points where $\rank \phi_1 = a_1+c_1$, so the choice of flag $R \subset S$ as above is unique: take $R = \alpha_1(A_1)$ and $S = \phi_1(A_1 \oplus C_1)$. In particular, it is dense since we have shown that $Z$ is irreducible. We can conclude that
  \begin{align*}
    \dim Z &= \dim \wt{Z}\\
           &= \dim \bF + \rank (\hom(A_1, \cR) \oplus \hom(C_1, \cS) \oplus \hom(B/\cS, A_2) \oplus \hom(B/\cR,C_2)) \\
           &= a_1c_1 + (a_1+c_1)(b-a_1-c_1) + a_1^2 + c_1(a_1+c_1) + (b-a_1-c_1)a_2 + (b-a_1)c_2\\
           &= (a_1b + c_1b + ba_2 + bc_2) - a_1a_2 - c_1a_2 - a_1c_2.
  \end{align*}
  In particular, $\codim (Z,\bX) = a_1a_2 + c_1a_2 + a_1c_2$. But $Z$ is defined by $a_1a_2+c_1a_2+a_1c_2$ many equations (the entries of the compositions $\alpha_2\alpha_1$, $\gamma_2\alpha_1$, and $\alpha_2\gamma_1$), so $Z$ is a complete intersection; let $I$ be its defining ideal.

  \medskip
  
  {\bf Step 2:} Let $U$ be the open subset of $Z$ consisting of points such that
  \begin{itemize}
  \item $\rank \phi_1 = a_1+c_1$, or
  \item $\rank \phi_2 = a_2+c_2$.
  \end{itemize}
  If $x \in U$, we claim that the Jacobian matrix of $I$ at $x$ has maximal rank $a_1a_2 + c_1a_2 + a_1c_2$.
  
  In the first case, up to a change of basis we may assume that $\phi_1 = \begin{bmatrix}{\rm id}_{a_1+c_1} \\ 0\end{bmatrix}$ and by semicontinuity, we may assume that $\phi_2=0$ (i.e., by considering the degeneration obtained by scaling the entries in $\phi_2$). For this choice, the Jacobian matrix has $a_1a_2+a_1c_2+c_1a_2$ 1's (each in different rows and columns) and 0's elsewhere. The argument for the case when $\rank \phi_2 = a_2+c_2$ is similar.

  \medskip
  
{\bf Step 3:}  Let $Z^0 = Z \setminus U$. We claim that $\codim(\pi^{-1}(Z^0), \tilde{Z}) \ge 2$. This is defined by 2 rank conditions: both $\phi_1$ and $\phi_2$ fail to have maximal rank. Since both of them generically obtain maximal rank, and they involve disjoint sets of variables, each condition has codimension 1. In particular $\codim (Z^0, Z) \ge 2$ as well. Since $Z^0$ contains the singular locus of $Z$, we conclude that $Z$ is regular in codimension 2, and hence is reduced and normal. In particular, we can apply \cite[Proposition 4.1]{superres} to conclude that $Z$ has rational singularities.
\end{proof}

In the special case $c_1=c_2=0$, $Z$ is an instance of the variety of complexes (see \cite{dcs} for more information) and the previous result is well-known. We record it separately for convenience.

\begin{corollary} \label{cor:complex-CI}
  Let $\cA_1,\cA_2,\cB$ be vector bundles over $X$ of ranks $a_1,a_2,b$, respectively, and assume that $a_1+a_2 \le b$. Let $Z$ be the zero locus of the tautological map $\cA_1 \to \cA_2$ over the total space of $\hom(\cA_1, \cB) \oplus \hom(\cB, \cA_2)$.

  Then the map $Z \to X$ is flat, and all of the fibers are reduced, irreducible, locally complete intersections, and have rational singularities. In particular, if $X$ is reduced, respectively, irreducible, locally a complete intersection, normal, or has rational singularities, then the same holds for $Z$.
\end{corollary}

\begin{remark} \label{rmk:complex-CI}
  We can relax the inequality slightly and consider the case $b=a_1+a_2-1$. Then $Z$ is still a complete intersection and reduced, but it is reducible since $\rank \alpha_1 + \rank \alpha_2 \le b$. In particular, it has two irreducible components, one defined by the condition $\rank \alpha_1 < a_1$ and the other defined by the condition $\rank \alpha_2 < a_2$. The intersection of these components has codimension 1, so $Z$ fails to be regular in codimension 2. The relevant dimension counts can be done using vector bundle desingularizations as in the previous proof.
\end{remark}

We mention one more corollary in the case that $c_1=c_2=0$. Consider the Grassmannian $\Gr(p,V)$ with tautological exact sequence
\[
  0 \to \cR \to V \times \Gr(p,V) \to \cQ \to 0.
\]
  By functoriality and basic properties of Schur functors, for partitions $\alpha$ and $\beta$, we have canonical surjections
  \[
    \bS_\alpha V \times \Gr(p,V) \to \bS_\alpha \cQ \to 0, \qquad \bS_\beta V^* \times \Gr(p,V) \to \bS_\beta(\cR^*) \to 0,
  \]
  and hence a canonical surjection
  \begin{align} \label{eqn:schur-surj1}
    (\bS_\alpha V \otimes \bS_\beta (V^*)) \times \Gr(p,V) \to \bS_\alpha \cQ \otimes \bS_\beta(\cR^*) \to 0.
  \end{align}

\begin{corollary} \label{cor:rat-schur-surj}
  The map on sections induced by \eqref{eqn:schur-surj1} is surjective.
\end{corollary}

\begin{proof}
  We consider the proof of Proposition~\ref{prop:complex-CI} with $X$ a point, $c_1=c_2=0$, $a_1=p$, $a_2=b-p$, and $B = V$. Then $\bF = \Gr(p,V)$ and $\eta = (A_1 \otimes \cR^*) \oplus (\cQ \otimes A_2^*)$. The projection map $\pi \colon \Spec_\bF(\Sym(\eta)) \to \hom(A_1, V) \times \hom(V, A_2)$ is birational, and by Proposition~\ref{prop:complex-CI}, the image is a normal variety. Hence $\rH^0(\bF; \Sym(\eta))$ gives the coordinate ring of the image, which implies that the map on sections of the canonical map
  \[
    \Sym( (A_1 \otimes V^*) \oplus (V \otimes A_2^*)) \to \Sym(\eta)
  \]
  is surjective. By the Cauchy identity and semisimplicity, taking the $\bS_\beta(A_1) \otimes \bS_\alpha(A_2^*)$ isotypic component with respect to the action of $\GL(A_1) \times \GL(A_2)$ gives the desired result.
\end{proof}

\begin{remark} \label{rmk:pos-char}
  While our primary focus in this paper uses fields of characteristic 0, we remark that everything in this section can be done over an algebraically closed field of any characteristic, except possibly for the statement about rational singularities.
\end{remark}

\section{Rational Schur functors} \label{sec:rational-schur}

\subsection{Construction}
Let $E,F$ be vector spaces and let $V=V_0|V_1$ be a super vector space. Define
\[
  W = \hom(E,V) \oplus \hom(V,F), \qquad A = \Sym(W^*).
\]
We think of $W$ as a supervariety and $A$ as its coordinate ring.
As before, we can represent a general point of $W$ as a pair of maps $(f,g)$ and the entries of the composition $gf$ form the space $E \otimes F^*$; let $(E \otimes F^*)$ be the ideal in $A$ that this space generates.

\begin{proposition} \label{prop:flat-super}
  Suppose that $\dim V_0 \ge \dim E + \dim F - 1$. Then $(E \otimes F^*)$ is a complete intersection in the sense that it is resolved by the Koszul complex on  $E \otimes F^*$.
\end{proposition}

\begin{proof}
  We consider a 1-parameter family where all odd variables in $A$ are scaled by $t$. Let $I'$ be the ideal generated by $E \otimes F^*$ at $t=0$. Then we are in the situations of Corollary~\ref{cor:complex-CI} and Remark~\ref{rmk:complex-CI} tensored with an exterior algebra on $\hom(E,V_1)^* \oplus \hom(V_1,F)^*$. The result follows by repeated applications of the following claim.

  Suppose that $B$ is a graded supercommutative algebra with finite-dimensional graded components, and consider the 1-parameter family where all odd generators are scaled by $t$. Let $f \in B$ be homogeneous and let $f_0$ be the result of substituting $t=0$. If $f_0$ is a nonzerodivisor, we claim that $f$ is also a nonzerodivisor. This follows by considering the rank of the multiplication map $f_t \colon B_d \to B_{d + \deg(f)}$ for all $d$: $\rank f_t \ge \rank f_0$ for all $t$, and $f_0$ being a nonzerodivisor is equivalent to multiplication by $f_0$ being injective, i.e., $\rank f_0$ being maximal for all degrees $d$.
\end{proof}

Consider the natural action of $\GL(E) \times \GL(F)$ on $A/(E \otimes F^*)$. This commutes with the action of $\GL(V)$. We denote the $\bS_\mu(E) \otimes \bS_\lambda(F^*)$ multiplicity space by $\bS_{[\lambda;\mu]}(V)$, and call the construction $V \mapsto \bS_{[\lambda;\mu]}(V)$ a \defi{rational Schur functor}. It follows from the Cauchy identity that when $\mu=0$, this is the usual Schur functor $\bS_{[\lambda;0]}(V) = \bS_\lambda(V)$, and when $\lambda=0$, we get the dual space $\bS_{[0;\mu]}(V) = (\bS_\mu(V))^*$. The Koszul complex gives an exact sequence
  \begin{align*}
  \cdots \to \bigwedge^i(E \otimes F^*) \otimes A \to \cdots \to (E \otimes F^*) \otimes A \to A \to A/(E \otimes F^*) \to 0
  \end{align*}
  and taking the $\bS_\mu(E) \otimes \bS_\lambda(F^*)$ multiplicity space gives an exact sequence
\begin{align} \label{eqn:rat-schur-littlewood}
    \cdots \to \bigoplus_{|\nu|=i} \bS_{\lambda/\nu}(V) \otimes \bS_{\mu/\nu^T}(V^*) \to \cdots \to \bS_{\lambda/1} V \otimes \bS_{\mu/1} (V^*) \to \bS_\lambda V \otimes \bS_\mu(V^*) \to \bS_{[\lambda;\mu]}(V) \to 0
\end{align}

\begin{proposition} \label{prop:rational-schur-coord}
    As representations of $\GL(V_0) \times \GL(V_1)$, we have:
  \begin{align*}
    A / (E \otimes F^*) &\cong  \bigoplus_{\lambda,\mu} \bS_\mu E \otimes \bS_{[\lambda;\mu]}(V) \otimes \bS_\lambda(F^*)  \\
    &\cong \left( \bigoplus_{\lambda,\mu} \bS_\mu E \otimes \bS_{[\lambda;\mu]}(V_0) \otimes \bS_\lambda(F^*) \right) \otimes \bigwedge^\bullet(E \otimes V_1^*) \otimes \bigwedge^\bullet(V_1 \otimes F^*).
  \end{align*}
  and
  \begin{align*}
    \bS_{[\lambda;\mu]} V  &= \bigoplus_{\alpha,\beta,\gamma,\delta} (\bS_{[\alpha;\beta]}(V_0) \otimes \bS_{\gamma}(V_1^*) \otimes \bS_\delta(V_1))^{\oplus c^\mu_{\beta, \gamma^T} c^\lambda_{\alpha, \delta^T}}
  \end{align*}
  where $c$ denotes the Littlewood--Richardson coefficient.
\end{proposition}

\begin{proof}
  The first isomorphism holds by definition of $\bS_{[\lambda;\mu]}(V)$. The second isomorphism holds by considering the flat family constructed in the proof of Proposition~\ref{prop:flat-super}. The last formula holds by comparing these two isomorphisms and taking the $\bS_\mu E \otimes \bS_\lambda(F^*)$-isotypic component. We compute this for the last expression using the Cauchy identity and the definition of Littlewood--Richardson coefficients:
  \begin{align*}
    & \left( \bigoplus_{\alpha,\beta} \bS_\beta E \otimes \bS_{[\alpha;\beta]}(V_0) \otimes \bS_\alpha(F^*) \right) \otimes \bigwedge^\bullet(E \otimes V_1^*) \otimes \bigwedge^\bullet(V_1 \otimes F^*)\\
    =& \left( \bigoplus_{\alpha,\beta} \bS_\beta E \otimes \bS_{[\alpha;\beta]}(V_0) \otimes \bS_\alpha(F^*) \right) \otimes \left(\bigoplus_\gamma \bS_{\gamma^T}E \otimes \bS_\gamma(V_1^*) \right) \otimes  \left( \bigoplus_\delta \bS_\delta(V_1) \otimes \bS_{\delta^T}(F^*) \right)\\
    =&\bigoplus_{\alpha,\beta,\gamma,\delta} (\bS_\mu(E) \otimes \bS_{[\alpha;\beta]}(V_0) \otimes \bS_{\gamma}(V_1^*) \otimes \bS_\delta(V_1) \otimes \bS_\lambda(F^*) )^{\oplus c^\mu_{\beta, \gamma^T} c^\lambda_{\alpha, \delta^T}}. \qedhere
  \end{align*}
\end{proof}

\begin{remark}
  In the last expression above, the $\GL(V_1)$-representations $\bS_\gamma(V_1^*) \otimes \bS_\delta(V_1)$ are not irreducible in general, and can be rewritten as a direct sum of rational Schur functors on $V_1$. However, this adds a significant complication to the formula, while the version that we have is sufficient for our applications.
\end{remark}

We end with a determinantal formula for the character of $\bS_{[\lambda;\mu]}(V)$; we will use brackets $[-]$ to denote taking the character of a representation. This is probably not new, but  we will use it shortly, and we could not find a clear reference, so we will give a proof.

Pick $p \ge \ell(\lambda)$ and $q \ge \ell(\mu)$. To avoid small indices, let $\ol{h}(n) = [\Sym^n(V^*)]$ and $h(n) = [\Sym^n(V)]$. We first review the classical Jacobi--Trudi identity. This states that for partitions $\alpha \subseteq \beta$ and any $n \ge \ell(\beta)$, we have (we will let $i$ denote row indices, and $j$ denote column indices):
\begin{align} \label{eqn:J-T}
  [\bS_{\alpha/\beta}V] = \det( h (\alpha_j - \beta_i - j + i) )_{i,j=1}^n, \quad   [\bS_{\alpha/\beta}(V^*)] = \det( \ol{h} (\alpha_j - \beta_i - j + i) )_{i,j=1}^n.
\end{align}
For symmetric functions, this identity appears as \cite[Theorem 7.16.1]{EC2} (although we have transposed the matrix for compatibility with a later formula). We have a symmetric monoidal functor from the category of polynomial functors to representations of $\GL(V)$ given by evaluation at $V$ (see for instance, \cite[(6.5.1)]{expos}). On Grothendieck rings, this induces a homomorphism from the ring of symmetric functions to the character ring of $\GL(V)$ which sends $h_k \mapsto [\Sym^k(V)]$ and $s_{\alpha/\beta} \mapsto [\bS_{\alpha/\beta}(V)]$; the second identity is just obtained by swapping $V$ for $V^*$.

Let $M$ be the $(p+q) \times q$ matrix given by
\[
  M_{i,j} = \ol{h}(\mu_{q+1-j} -i+j)
\]
In other words, on the main diagonal, we have the entries $\ol{h}(\mu_q), \dots, \ol{h}(\mu_1)$; when going up along a column, the indices increase by 1 each entry and when going down, they decrease by 1 each entry. We also define a $(p+q) \times p$ matrix $L$ by
\[
  L_{i,j} = h(\lambda_j + i - q - j).
\]
Finally, let $M|N$ denote the square $(p+q) \times (p+q)$ matrix obtained by concatenating $L$ to the right side of $M$. In particular, the diagonal entries of $M|L$ are $\ol{h}(\mu_q),\dots, \ol{h}(\mu_1), h(\lambda_1), \dots, h(\lambda_p)$.

The determinant $\det(M|L)$ is called a \defi{composite supersymmetric Schur polynomial} in \cite[\S 8.4]{comes-wilson}; see loc. cit. also for some small examples. We will use their work on these characters in a later section.

\begin{proposition} \label{prop:rat-schur-det}
  If $\dim V_0 \ge \ell(\lambda)+\ell(\mu)-1$, then
  \[
    [\bS_{[\lambda;\mu]}(V)] = \det(M|L).
  \]
\end{proposition}

\begin{proof}
  \addtocounter{equation}{-1}
  \begin{subequations}
  Let $S \subseteq \{1,\dots,p+q\}$ be a subset of size $q$ and let $s_1 > \cdots > s_q$ be its members written in decreasing order. We define a few quantities:
  \begin{itemize}
  \item $S^c = \{1,\dots,p+q\} \setminus I$; let $t_1 < \cdots < t_p$ be its members written in increasing order,
  \item $\Delta_S$ is the determinant of the $q\times q$ submatrix of $M$ using the rows indexed by $S$,
  \item $\nabla_S$ is the determinant of the $p \times p$ submatrix of $L$ using the rows indexed by $S^c$,
  \item $\gamma(S) = (s_1 - q, s_2 - q + 1, \dots, s_q - 1)$, which is a partition, and
  \item $\delta(S) = (q+1-t_1, q+2-t_2, \dots, q+p-t_p)$, which is also a partition.
  \end{itemize}

  By Laplace expansion, we have
  \begin{align} \label{eqn:laplace}
    \det(M|L) = \sum_S (-1)^{|\gamma(S)|} \Delta_S \nabla_S
  \end{align}
  where the sum is over all subsets $S \subseteq\{1,\dots,p+q\}$ of size $q$. Now we identify the terms on the right. Using \eqref{eqn:J-T}, $\Delta_S = [\bS_{\mu/\gamma(S)}(V^*)]$ and $\nabla_S = [\bS_{\lambda/\delta(S)}(V)]$.

  Next, we claim that $\delta(S)=\gamma(S)^T$. Note that $\gamma(S)^T$ is the unique partition such that for all $a,b \ge 1$, we have $(\gamma(S)_a \ge b) \iff (\gamma(S)^T_b \ge a)$, so it suffices to show that $\delta(S)$ satisfies these equivalences. So pick $a,b \ge 1$. Then
  \begin{align*}
    \gamma(S)_a \ge b&\iff s_a \ge q+1-a+b\\
                     &\iff |S \cap \{1,\dots,q-a+b\}| \le q-a\\
                     &\iff |S^c \cap \{1,\dots,q-a+b\}| \ge b\\
                     &\iff t_b \le q-a+b\\
    &\iff \delta(S)_b \ge a,
  \end{align*}
  which proves the claim. The function $S \mapsto \gamma(S)$ is a bijection between the set of $q$-element subsets of $\{1,\dots,p+q\}$ and the set of partitions $\gamma$ such that $\gamma \subseteq q \times p$, i.e., that satisfy $\ell(\gamma) \le q$ and $\gamma_1 \le p$. Hence \eqref{eqn:laplace} can be written as
  \[
    \det(M|N) = \sum_{\gamma \subseteq q \times p} (-1)^{|\gamma|} [\bS_{\mu/\gamma}(V^*)] [\bS_{\lambda/\gamma^T}(V)].
  \]
  The left side is the Euler characteristic of \eqref{eqn:rat-schur-littlewood} (without the rightmost term) since if $\gamma \not\subseteq q \times p$, then either $\bS_{\mu/\gamma}(V^*)=0$ or $\bS_{\lambda/\gamma^T}(V)=0$, and so it simplifies to $[\bS_{[\lambda;\mu]}(V)]$, as desired.
\end{subequations}
\end{proof}

\subsection{Irreducibility} \label{ss:irred}

Our next goal is to present some cases in which $\bS_{[\lambda;\mu]}(V)$ is an irreducible representation of $\GL(V)$ (or equivalently of its Lie superalgebra $\fgl(V)$).

First we need to discuss some generalities on highest weights. Unlike the classical case, there are multiple conjugacy classes of Borel subalgebras in $\fgl(V)$. To pick one, let $v_1,\dots,v_{m+n}$ be an ordered basis for $V$ such that $v_1,\dots,v_m \in V_0$ and $v_{m+1},\dots,v_{m+n} \in V_1$. We consider the Borel subalgebra $\fb$ of upper-triangular matrices and the Cartan subalgebra $\fh$ of diagonal matrices with respect to this ordered basis.

Let ${\rm diag}(x_1,\dots,x_{m+n})$ denote the diagonal matrix with entries $x_1,\dots,x_{m+n}$. The sequence $(a_1,\dots,a_m|a_{m+1},\dots,a_{m+n})$ denotes the weight $\fh \to \bC$ given by ${\rm diag}(x_1,\dots,x_{m+n}) \mapsto \sum_i a_i x_i$.

Next, let's fix some notation we will use throughout this section:
\[
  r = \ell(\lambda), \qquad s = \ell(\mu), \qquad t = \max\{i \mid \mu_i > n\} = \mu_{n+1}^T.
\]
With these choices, the highest weight of $\bS_\lambda(V)$ is
\[
  \bw(\lambda) = (\lambda_1,\dots,\lambda_r,0,\dots,0|0,\dots,0).
\]
To describe the highest weight of the dual $\bS_\mu(V^*)$, let $\beta = (\beta_1,\dots,\beta_t)$ be the result of removing the first $n$ columns from the Young diagram of $\mu$, i.e., $\beta_i = \mu_i - n$ for $i=1,\dots,t$. Then the highest weight of $\bS_\mu(V^*)$ is
\[
  \bw(\mu^*) = (0,\dots,0, -\beta_t,\dots,-\beta_1 | {-\mu_n^T}, \dots, -\mu^T_1).
\]
In particular, the tensor product of the highest weight vectors in $\bS_\lambda(V)$ and $\bS_\mu(V^*)$ gives a highest weight vector in $\bS_\lambda(V) \otimes \bS_\mu(V^*)$ of weight
\[
  \bw(\lambda;\mu) := \bw(\lambda) + \bw(\mu^*).
\]

  \begin{theorem} \label{thm:schur-irred}
    If $\dim V_0 - \dim V_1 \ge \ell(\lambda) + \ell(\mu)$, then $\bS_{[\lambda;\mu]}(V)$ is an irreducible representation of either $\GL(V)$ or $\fgl(V)$.
  \end{theorem}

  As a first step, we will use the results in \cite{comes-wilson} to deduce that the representation has the same character as an indecomposable representation.

    \begin{lemma}
    If $\dim V_0 - \dim V_1 \ge \ell(\lambda) + \ell(\mu)$, then there exists an indecomposable $\GL(V)$-representation $W(\lambda,\mu)$ with the same character as $\bS_{[\lambda;\mu]}(V)$.
  \end{lemma}
  
  \begin{proof}
    We will use the notation of \cite[\S 6.3]{comes-wilson} so $\delta = \dim V_0 - \dim V_1$ (but note that loc. cit. uses $\lambda$ to denote a pair of partitions). In particular, treat $\lambda$ and $\mu$ as infinite sequences with $\lambda_i=0$ for $i>\ell(\lambda)$, and similarly for $\mu$, and define
    \begin{align*}
      I_\wedge &= \{\lambda_i - i + 1 \mid i=1,2,3,\dots\},\\
      I_\vee &= \{ i - \delta - \mu_i \mid i=1,2,3,\dots\}.
    \end{align*}
    Let $L = I_\wedge \setminus I_\vee$ and $M = I_\vee \setminus I_\wedge$. We claim that $\max(L) < \min(M)$.
    First, 
    \[
      (-\infty,-r] := \{-r,-r-1,-r-2,\dots\} \subset I_\wedge,
    \]
    and
    \[
      [s-\delta+1,\infty) := \{s-\delta+1, s-\delta+2, \dots\} \subset I_\vee.
    \]

    Note that
    \[
      \{\lambda_i - i + 1 \mid i=1,\dots,r\} \subset [s-\delta+1, \infty)
    \]
    since the smallest member of the first set is $\lambda_r - r + 1$, and $\lambda_r - r + 1 \ge s -\delta + 1$ by our assumption. Similarly, we have
    \[
      \{i-\delta-\mu_i \mid i=1,\dots,s \} \subset (-\infty, -r].
    \]
    In particular, we conclude that
    \[
      L = (-\infty,-r] \setminus I_\vee, \qquad M = [s-\delta+1,\infty) \setminus I_\wedge.
    \]
Set $d = \delta-r-s$, so that $d \ge 0$ by our assumption. Since the $d$ smallest elements of $[s-\delta+1,\infty)$ are the same as the $d$ largest elements of $(-\infty,-r]$, we get $\min(M) \ge -r+1$ and $\max(L) \le -r - d$, which proves the claim.

    Again in the notation of \cite[\S 6.3]{comes-wilson}, this claim implies that the cap diagram of $(\lambda,\mu)$ has no caps, and hence $D_{(\lambda,\mu),(\lambda',\mu')} = 1$ if $\lambda=\lambda'$ and $\mu=\mu'$, and is 0 otherwise (combining \cite[Corollary 6.4.2]{comes-wilson} with the definition of the numbers $D'$ in terms of cap diagrams).
    Now let $W(\lambda,\mu)$ be the indecomposable $\GL(V)$-representation constructed in \cite[\S 8.3]{comes-wilson}. By \cite[Theorem 8.5.2]{comes-wilson}, the above calculation shows that the character of $W(\lambda,\mu)$ is given by the composite supersymmetric Schur polynomial $s_{(\lambda,\mu)}$. By Proposition~\ref{prop:rat-schur-det}, this agrees with the character of $\bS_{[\lambda;\mu]}(V)$.
  \end{proof}

  To finish the proof, we will show that $W(\lambda,\mu)$ is irreducible; by linear independence of irreducible characters, this will prove that $W(\lambda,\mu) \cong \bS_{[\lambda;\mu]}(V)$. To do this, we will use \cite[Theorem 2.30]{cheng-wang} which classifies which weights have the same central character. Since irreducible representations with different central characters cannot have extensions, to prove that $W(\lambda,\mu)$ is irreducible, it will suffice to show that any potential subrepresentation of $W(\lambda,\mu)$ must have a different central character from the irreducible representation with highest weight $\bw(\lambda;\mu)$.

  Now we explain its statement. Weights are elements in the vector space $\fh^* \cong \bC^{m+n}$; write the standard basis as $e_1,\dots,e_m,f_1,\dots,f_n$. This is equipped with the symmetric bilinear form $(,)$ induced by supertrace given by
    \[
      (e_i,e_j) = \delta_{i,j}, \qquad (e_i,f_j) = 0, \qquad (f_i, f_j) = -\delta_{i,j}.
    \]
    Odd isotropic roots are of the form $e_i - f_j$. Define
    \[
      \rho = (m,m-1,\dots,1|-1,-2,\dots,-n).
    \]
    Given dominant weights $\nu, \eta$ of $\fgl(V_0)$ and $\fgl(V_1)$, respectively, the irreducible representation with highest weight $(\nu|\eta)$ has the same central character as $\bw(\lambda;\mu)$ if and only if there exist mutually orthogonal odd roots $\alpha_1,\dots,\alpha_p$ such that $(\bw(\lambda;\mu)+\rho, \alpha_i)=0$ and
    \begin{align} \label{eqn:linked}
      (\nu|\eta) + \rho = \sigma( \bw(\lambda;\mu) + \rho + \sum_{i=1}^p c_i \alpha_i)
    \end{align}
    for some complex numbers $c_1,\dots,c_p$ and some permutation $\sigma \in \fS_m \times \fS_n$.

    We note that if $e_i - f_j$ and $e_{i'} - f_{j'}$ are orthogonal (and unequal), then $i \ne i'$ and $j \ne j'$. Furthermore, the condition $(( \alpha|\beta), e_i-f_j)=0$ simply means that $\alpha_i = - \beta_j$.

In the proof below, we will use some properties about Littlewood--Richardson coefficients and dominance order. The relevant definitions and references can be found in \S\ref{ss:linear-algebra}.
    
    \begin{proof}
      We claim that if $\bS_{[\nu^+;\nu^-]}(V_0) \otimes \bS_{[\eta^+;\eta^-]}(V_1)$ appears in $W(\lambda,\mu)$ and \eqref{eqn:linked} holds, then $\nu^+=\lambda$, $\nu^-=\beta$, and $\eta^+=\emptyset$, and $\eta^- = (\mu_1^T,\dots,\mu_n^T)$. Write $(\nu|\eta)$ for the highest weight of this representation.
      First, we have
      \[
        \bw(\lambda;\mu) + \rho = (\lambda_1 + m, \dots, \lambda_r + m + 1 - r, m - r, \dots, t+1, t-\beta_t, \dots, 1-\beta_1 | -\mu_n^T-1,\dots, -\mu_1^T - n).
      \]
      Note that both sides are strictly decreasing sequences and that $\mu_n^T+1 \ge t+1$ and $\mu_1^T + n \le m-r$. In particular, if $e_i - f_j$ is orthogonal to $\bw(\lambda;\mu)+\rho$, then $r+1 \le i \le m-t$. Suppose that $(\nu|\eta)$ satisfies \eqref{eqn:linked}. For short, define
      \[
        \bu := \bw(\lambda;\mu) + \rho + \sum_i c_i \alpha_i.
      \]
      First, since all of our weights are integral, and the $\alpha_i$ have no overlapping nonzero components, the $c_i$ are all integers.
      
      The effect of applying the permutation $\sigma$ is to sort the vector $\bu$. Considering the first $m-t$ entries, since adding the $\alpha_i$ only affects the entries after $\lambda$ and before $\beta$, adding any positive, respectively negative, multiples of orthogonal roots $e_i-f_j$ will result in $|\nu^+|>|\lambda|$, respectively $|\nu^-|>|\beta|$. But if $\bS_{[\nu^+;\nu^-]}(V_0) \otimes \bS_{[\eta^+;\eta^-]}(V_1)$ appears in $W(\lambda,\mu)$, then $\nu^+ \subseteq \lambda$ using the formula in Proposition~\ref{prop:rational-schur-coord} (since the Littlewood--Richardson coefficient $c^\lambda_{\nu^+, \delta^T}$ must be nonzero for some partition $\delta$) and hence we must have $\nu^+=\lambda$, which means that $c_i \le 0$ for all $i$. The same formula also implies that $\eta^+=\emptyset$ and $c^\mu_{\nu^-,(\eta^-)^T} \ne 0$.
      
      Suppose that $\sum_i c_i \alpha_i \ne 0$ in \eqref{eqn:linked}. Let $J$ be maximal such that a root of the form $e_I - f_J$ appears in the sum and write $j = n+1-J$. Since $e_I-f_J$ is orthogonal to $\bw(\lambda;\mu)+\rho$, we have $I=m + 1 - \mu_j^T - J$. As explained above, we have $c_i \le 0$ for all $i$.

The last $j-1$ entries of the vector $\sum_i c_i \alpha_i$ are 0, so when we sort the last $m$ entries of $\bu$ the last $j-1$ entries are unaffected; we conclude that
$(\eta^-)^T_i =  \mu_i$ for all $i > \mu_j^T$. On the other hand, since $\sum_i c_i \alpha_i$ is negative in position $e_I$, when sorting $\bu$, the entry in the $I$th position either stays fixed or moves to the right and is replaced by something that was originally to the right of it. In either case, we conclude that $\nu^-_{m+1-I} > 0 = \beta_{m+1-I}$. By maximality of $J$, all of the entries to the left of position $e_I$ in $\sum_i c_i \alpha_i$ are 0 and so are unaffected by the sorting process. Putting these together gives the inequality
      \[
        \sum_{i \ge \mu_j^T + J} (\eta^-)^T_i + \nu^-_i > \sum_{i \ge \mu_j^T + J} \mu_i.
      \]
      Since $|(\eta^-)^T|+|\nu^-| = |\mu|$, this implies that $(\eta^-)^T + \nu^- \not\ge \mu$ (dominance order), which directly contradicts $c^\mu_{\eta^-, (\nu^-)^T} \ne 0$.  In particular, in the relation \eqref{eqn:linked}, we must have $\sum_i c_i \alpha_i = 0$; but then $\bu$ is already sorted, so we conclude that $\eta^- = (\mu_1^T,\dots,\mu_n^T)$ and $\nu^-=\beta$, which proves our initial claim.

Finally, since $W(\lambda,\mu)$ is indecomposable, each of its composition factors has the same central character. The highest weight for a composition factor must be a highest weight with respect to the $\fgl(V_0) \times \fgl(V_1)$ decomposition. The claim implies that $W(\lambda,\mu)$ only has composition factors of highest weight $\bw(\lambda;\mu)$; since this appears with multiplicity 1, we can conclude that $W(\lambda,\mu)$ is irreducible. As stated above, we can also conclude that $W(\lambda,\mu) \cong \bS_{[\lambda;\mu]}(V)$.
\end{proof}

\subsection{Additional remarks}
 
  \begin{remark}
    The definition of $\bS_{[\lambda;\mu]}(V)$ makes sense more generally; we can define it as the cokernel of the canonical map $\bS_{\lambda/1} V \otimes \bS_{\mu/1}(V^*) \to \bS_\lambda V \otimes \bS_\mu(V^*)$. However, it need not be irreducible in general. For instance, if $\lambda=\mu=(1)$, then $\bS_{[1;1]}(V)$ can be identified with $\fp\gl(V)$. If $\dim V_0=\dim V_1>1$, then supertrace gives a well-defined nonzero equivariant map $\fp\gl(V) \to \bC$, so that its kernel defines a proper subrepresentation of $\fp\gl(V)$.
  \end{remark}

  \begin{remark}
See \cite{DS,DS2} for more general information.

  Let $\fg$ be a finite-dimensional Lie superalgebra, $x \in \fg_1$ such that $[x,x]=0$, and let $M$ be a $\fg$-module. Let $x_M$ denote the operator $m \mapsto x\cdot m$ on $M$. Then $x_M^2=0$, so we can define $M_x = \ker x_M / \im x_M$. Letting $\fg$ act on itself by the adjoint representation, we can also construct $\fg_x$; this is a Lie superalgebra which acts on $M_x$ \cite[Lemma 2.2]{DS2}. The assignment $M \mapsto M_x$ gives a functor
  \[
    \ds_x \colon \fg\text{-mod} \to \fg_x\text{-mod},
  \]
  which is called the \defi{Duflo--Serganova functor}. In fact, $\ds_x$ is a symmetric monoidal functor and commutes with taking duals \cite[Lemma 2.4]{DS2}. If $0 \to M_1 \to M_2\to M_3 \to 0$ is a short exact sequence of $\fg$-modules, then $(M_1)_x \to (M_2)_x \to (M_3)_x$ is exact, and $(M_1)_x \to (M_2)_x$ is injective if and only if $(M_2)_x \to (M_3)_x$ is surjective \cite[Lemma 2.7]{DS2}.

  If $\fg = \fgl(m|n)$, then $\fg_x \cong \fgl(m-r|n-r)$ for some $r \ge 0$ which we will denote by $r = \rank (x)$.  We claim that if $\ell(\lambda)+\ell(\mu) -1 \le m-r$, then $\ds_x(\bS_{[\lambda;\mu]}(\bC^{m|n})) = \bS_{[\lambda;\mu]}(\bC^{m-r|n-r})$ for all $x$ with $\rank(x)=r$. Since $\ds_x$ is symmetric monoidal, it commutes with constructing Schur functors. So the claim follows by using the exactness property mentioned earlier and comparing the Koszul complexes \eqref{eqn:rat-schur-littlewood} for $V = \bC^{m|n}$ and $V = \bC^{m-r|n-r}$.

  It is interesting to compare this to \cite[Theorem 12.14]{DS2} which computes the effect of $\ds_x$ on an irreducible representation using arc diagrams. If we assume that $m-n \ge \ell(\lambda)+\ell(\mu)$, so that $\bS_{[\lambda;\mu]}(\bC^{m|n})$ is irreducible from our results above, then the corresponding arc diagram for the weight $\bw(\lambda;\mu)$ has $n$ arcs, which are totally nested, i.e., we can order the arcs so that the $i$th arc is contained inside the $(i+1)$st arc. Removing the maximal arc gives the weight $\bw(\lambda;\mu)$ (but for $\fgl(m-1|n-1)$) so this gives an alternate way to see the compatibility of this construction with $\ds_x$.
\end{remark}

\section{Main theorem} \label{sec:main-thm}

\subsection{Statement and outline of proof} \label{ss:outline}
For now, we focus on the super Grassmannian $X = \Gr(p|q, V)$. This has a tautological short exact sequence
\[
  0 \to \cR \to V \times X \to \cQ \to 0
\]
where $\cR$ is the tautological subbundle of rank $p|q$ and $\cQ$ is the tautological quotient bundle of rank $m-p|n-q$.

Generalizing~\eqref{eqn:schur-surj1}, for partitions $\alpha$ and $\beta$, we have a canonical surjection
\begin{align} \label{eqn:schur-surj}
    (\bS_\alpha V \otimes \bS_\beta (V^*)) \times X \to \bS_\alpha \cQ \otimes \bS_\beta(\cR^*) \to 0.
  \end{align}
  
\begin{theorem} \label{thm:main-Q}
  Let $\alpha, \beta$ be partitions and assume that one of the following conditions holds:
  \begin{enumerate}
  \item $m-n-\ell(\alpha) \ge p -q \ge \ell(\beta)$, or
    
  \item $n-m -\alpha_1\ge q-p \ge \beta_1$.
  \end{enumerate}
In either case, the map \eqref{eqn:schur-surj} induces a surjection on cohomology and a $\GL(V)$-equivariant isomorphism of graded $\rH^\bullet(X;\cO_X)$-modules
  \[
    \rH^\bullet(X; \bS_\alpha \cQ \otimes \bS_\beta (\cR^*)) \cong \rH^\bullet(X; \cO_X) \otimes \bS_{[\alpha; \beta]}(V).
  \]
  Here $\GL(V)$ acts trivially on the vector space $\rH^\bullet(X;\cO_X)$, and the right side denotes the free $\rH^\bullet(X;\cO_X)$-module with generators $\bS_{[\alpha;\beta]}(V)$.

  Furthermore, $\bS_\alpha \cQ \otimes \bS_\beta (\cR^*)$ is $\cJ$-formal.
\end{theorem}

We first show that case (2) of the theorem follows from case (1).

\begin{proof}[Proof that case (1) implies case (2)]
  Let $V[1]$ denote the shifted superspace with $V[1]_0=V_1$ and $V[1]_1=V_0$. Then we have an isomorphism $\Gr(p|q,V) \cong \Gr(q|p,V[1])$ such that $\bS_\alpha (\cQ) \otimes \bS_\beta(\cR^*)$ gets identified with $\bS_{\alpha^T}(\cQ) \otimes \bS_{\beta^T}(\cR^*)$ on $\Gr(q|p,V[1])$. Since $\beta_1 = \ell(\beta^T)$ and $\alpha_1 = \ell(\alpha^T)$, case (1) implies that
  \[
    \rH^\bullet(X; \bS_\alpha \cQ \otimes \bS_\beta (\cR^*)) \cong \rH^\bullet(X; \cO_X) \otimes \bS_{[\alpha^T; \beta^T]}(V[1]).
  \]
  But it follows from the construction that $\bS_{[\alpha^T;\beta^T]}(V[1]) \cong \bS_{[\alpha; \beta]}(V)$.
\end{proof}

{\bf Therefore, we will assume throughout that assumption (1) holds:}
\[
  m - n - \ell(\alpha) \ge p- q \ge \ell(\beta).
\]

By \cite[Theorem 1.2(a)]{superres}, we have an isomorphism of graded algebras
\[
  \rH^\bullet(X;\cO_X) \cong \rH^\bullet_{\rm sing}(\Gr_q(\bC^n); \bC),
\]
where the right side is the singular cohomology ring of $\Gr_q(\bC^n)$, which is concentrated in even degrees. See \S\ref{ss:fact} for some more information.

~

The proof involves several involved calculations which, a priori, are not connected, so we will give a brief overview of the plan before starting. We continue to use the notation from \S\ref{ss:supergeom}.

Using the $\cJ$-adic filtration on the sheaf $\bS_\alpha \cQ \otimes \bS_\beta(\cR^*)$, we get a $\GL(V_0) \times \GL(V_1)$-equivariant spectral sequence
\[
  \rH^k(X_\bos; \bigwedge^\bullet (\cJ/\cJ^2) \otimes \bS_\alpha(\cQ_0|\cQ_1) \otimes \bS_\beta(\cR_0^*|\cR_1^*)) \Longrightarrow \rH^k(X; \bS_\alpha \cQ \otimes \bS_\beta(\cR^*)).
\]
Our main goal is to compute the terms on the left hand side. As we have explained, the cohomology of $\bigwedge^\bullet(\cJ/\cJ^2)$ has a concrete geometric interpretation: these groups compute Tor of an algebra which can be described explicitly, so we can attempt to compute it by studying the corresponding algebraic variety (and this strategy was carried out in \cite{superres}). The presence of the extra terms $\bS_\alpha(\cQ_0|\cQ_1) \otimes \bS_\beta(\cR_0^*|\cR_1^*)$ means that the corresponding cohomology groups compute Tor of some module over the aforementioned algebra.

The geometric meaning is less clear, but we can take advantage of the Cauchy identity. We introduce auxiliary vector spaces $E$ and $F$. Then taking the sum over all partitions $\alpha,\beta$, we have
\begin{align*}
  \bigoplus_{\alpha} \bS_\alpha(\cQ_0|\cQ_1) \otimes \bS_\alpha E  &= \Sym(\cQ_0 \otimes E) \otimes \bigwedge^\bullet(\cQ_1 \otimes E),\\
  \bigoplus_\beta  \bS_\beta F \otimes \bS_\beta(\cR_0^*|\cR_1^*)  &= \Sym(F \otimes \cR_0^*) \otimes \bigwedge^\bullet(F \otimes \cR_1^*).
\end{align*}
Thus, by remembering the $\GL(E) \times \GL(F)$ action, we can package all of the calculations as computing the cohomology of one single sheaf of algebras:
\[
  \bigwedge^\bullet(\cJ/\cJ^2) \otimes \Sym(\cQ_0 \otimes E) \otimes \bigwedge^\bullet(\cQ_1 \otimes E) \otimes \Sym(F \otimes \cR_0^*) \otimes \bigwedge^\bullet(F \otimes \cR_1^*).
\]
While it may be possible to do this directly, we will take an indirect approach.

To incorporate the symmetric algebras, we extend our construction by adding $(\cQ_0 \otimes E) \oplus (F \otimes \cR_0^*)$ as functions, i.e., taking the total space of the dual vector bundle. The exterior algebras are less clear, but they also have a place: they appear as Koszul complexes, and we incorporate them by constructing appropriate complete intersections.

We use this insight to construct a scheme $Y$ over $X_\bos$ together with a map to an affine scheme
\[
  \phi \colon Y \to \Spec(S)
\]
such that the cohomology groups for the above algebra compute $\Tor_\bullet^S(\phi_* \cO_Y, \bC)$. To take advantage of this, we introduce one more algebra $\wt{S}$ such that $\phi$ factors as
\[
  Y \xrightarrow{\pi} \Spec(\wt{S}) \to \Spec(S).
\]
Notably, we will show that $\pi_* \cO_Y$ is a free module over a complete intersection in $\Spec(\wt{S})$, which makes it easy to compute $\Tor_\bullet^{\wt{S}}(\pi_* \cO_Y, \bC)$. Finally, transferring the results to $S$ is simple since $\wt{S}$ is free as an $S$-module.

The conclusion is that we get very concrete calculations of the cohomology groups that we were originally interested in. To organize the material we split the proof as follows:
\begin{enumerate}
\item The first part constructs all of the relevant spaces and algebras.
\item The second part analyzes the properties of the algebras $S$ and $\wt{S}$.
\item The third part proves some properties about $Y$.
\item The fourth and final part combines everything together.
\end{enumerate}

\subsection{Proof part 1: setup and constructions}

Let $E,F$ be vector spaces such that
\begin{align} \label{eqn:ineq}
  p -q \ge \dim F, \qquad m-n \ge \dim E + p - q.
\end{align}

Define
\begin{align*}
  W &= \hom(V_0,V_1) \times \hom(V_1,V_0),\\
  W^\re &= W \times \hom(V_0, E) \times \hom(F,V_0),\\
  \wt{W} &= \{(f,g,h_1,h_2) \in W^\re \mid h_1h_2=0\}.
\end{align*}
The notation $\re$ and tilde are meant to suggest that these spaces are enlargements of $W$, though there is no further formal meaning behind this. For a quick visual aid, $f,g,h_1,h_2$ will always denote linear maps between the following vector spaces:
\[
  f \colon V_0 \to V_1, \quad g \colon V_1 \to V_0, \quad h_1 \colon V_0 \to E, \quad h_2 \colon F \to V_0.
\]

If $U$ is an affine $\bC$-scheme, we write its coordinate ring as $\bC[U]$ an abuse of notation, we identify a coherent sheaf on an affine scheme with its module of sections. Now we define a few coordinate rings:
\begin{align*}
  S &= \bC[W] = \Sym(W^*),\\  
  T &= \bC[\hom(F,V_0) \times \hom(V_0,E)] / (h_1h_2 = 0),\\
  \wt{S} &= \bC[\wt{W}].
\end{align*}
In the second definition, $h_1$ is the universal map $V_0\to E$ and similarly for $h_2$, and the equations $h_1h_2=0$ mean to take the ideal generated by the entries of $h_1h_2$.

We think of $S$ and $T$ as subrings of $\wt{S}$ (they are also quotient rings). We equip $\wt{S}$ with a bigrading via $\deg(W^*)=(1,0)$ and $\deg(F \otimes V_0^*) = \deg(V_0 \otimes E^*) = (0,1)$. Note that we have the following vector space decomposition
\[
  \wt{S} = S \otimes T, \qquad \wt{S}_{(a,b)} = S_a \otimes T_b,
\]
and that $\wt{S}$ is both a free $S$-module and a free $T$-module.

Furthermore, since $m \ge \dim E + \dim F$, it follows from \S\ref{sec:rational-schur} that we have the following decomposition of $T$ as a representation of $\GL(E) \times \GL(V_0) \times \GL(F)$:
\begin{align} \label{eqn:T-decomp}
      T = \bigoplus_{\substack{\lambda, \mu\\ \ell(\lambda) \le \dim E, \ \ell(\mu) \le \dim F}} \bS_\lambda (E^*) \otimes \bS_{[\lambda; \mu]}(V_0) \otimes \bS_\mu F.
\end{align}

Next, we define some spaces over $X_\bos$:
\begin{align*}
  Y'' &= \{(R_0, R_1, f,g, h_1, h_2) \in X_{\bos} \times W^\re \mid h_1(R_0)=0, \ h_2(F) \subseteq R_0\},\\
  Y' &= \{(R_0,R_1,f,g,h_1,h_2) \in Y'' \mid f(R_0) \subseteq R_1, \ g(R_1) \subseteq R_0\},\\  
  Y &= \{(R_0,R_1,f,g,h_1,h_2) \in Y' \mid h_1g=0, \ fh_2=0\}.
\end{align*}
Both $Y''$ and $Y'$ are defined by linear conditions in $X_\bos \times W^\re$, so are vector bundles; we will elaborate more on this in \S\ref{ss:alg-up}.

If $(R_0,R_1,f,g,h_1,h_2) \in Y$, then it follows by definition that $h_1h_2=0$. This means that we have a projection map
\[
  \pi \colon Y \to \wt{W}.
\]

Finally, we will define a ``flagged'' version of $Y$, denoted $\cY$. Namely, a point of $\cY$ is a tuple ($R_{0,\bullet}, R_{1,\bullet}, f,g,h_1,h_2)$ where:
\begin{itemize}
\item $R_{0,\bullet}$ is a flag of $n+1$ subspaces of $V_0$ (with dimensions denoted by second entry in the subscript) $R_{0,p-q} \subset \cdots \subset R_{0,n+p-q}$,
\item $R_{1,\bullet}$ is a flag of $n+1$ subspaces of $V_1$ (with dimensions denoted by second entry in the subscript) $R_{1,0} \subset \cdots \subset R_{1,n}$ (note that $R_{1,0}=0$ and $R_{1,n}=V_1$),
\item $f \colon V_0 \to V_1$ is a linear map such that $f(R_{0,i+p-q}) \subseteq R_{1,i}$ for all $i$,
\item $g \colon V_1 \to V_0$ is a linear map such that $g(R_{1,i}) \subseteq R_{0,i+p-q}$ for all $i$,  
\item $h_1 \colon V_0 \to E$ is a linear map such that $h_1(R_{0,n+p-q})=0$, and
\item $h_2 \colon F \to V_0$ is a linear map such that $h_2(F)\subseteq R_{0,p-q}$.
\end{itemize}
Note that the conditions above imply that $h_1g=fh_2=h_1h_2=0$. In particular, $\cY$ is a vector bundle, unlike $Y$ which requires some non-linear conditions.

We also get a map
\[
  \rho \colon \cY \to \wt{W}.
\]

\subsection{Proof part 2: algebra downstairs}

We introduce the following objects:
\begin{itemize}
\item Let $Z$ be the reduced locus in $\wt{W}$ where $h_1g=0$ and $fh_2=0$ (we remind the reader that the condition defining $\wt{W}$ inside $W^\re$ is $h_1h_2=0$).
\item Let $\chi(u) \in \bC[Z][u]$ denote the characteristic polynomial of $fg$.
\item Let $\Delta \in \bC[Z]$ be the discriminant of $\chi(u)$.
  
\item Let $\wt{\cZ}=\Spec( \Split_{\bC[Z]}(\chi))$, see \S\ref{ss:fact} for the definition and basic properties.
\item Let $\wt{Z}= \Spec( \Fact^{(q,n-q)}_{\bC[Z]}(\chi))$, see \S\ref{ss:fact}.
\end{itemize}

\begin{proposition} \label{prop:piY-eqns}
  \begin{enumerate}
  \item We have $\pi(Y) = \rho(\cY) = Z$.
  \item The map $\rho \colon \cY \to Z$ factors through $\wt{\cZ}$.
  \item The map $\pi \colon Y \to Z$ factors through $\wt{Z}$.
  \end{enumerate}
\end{proposition}

\begin{proof}
(1)   It is clear that $\rho(\cY) \subseteq \pi(Y) \subseteq Z$, so it suffices to show that $Z \subseteq \rho(\cY)$.

  Let $U$ be the open subset of $Z$ defined by the conditions: 
  \begin{itemize}
  \item $\psi \colon V_0 \to E \oplus V_1$ given by $x \mapsto (h_1(x), f(x))$ is surjective,
  \item The $n$ eigenvalues of $fg \colon V_1 \to V_1$ are distinct and nonzero,
  \item $h_2 \colon F \to V_0$ is injective (this condition is not important for this proof, but we will use it later).
  \end{itemize}
  Then $U$ is non-empty: for concreteness, pick bases $w_1,\dots,w_{\dim F}$ for $F$, $x_1,\dots,x_m$ for $V_0$, $y_1,\dots,y_n$ for $V_1$, and $z_1,\dots,z_{\dim E}$ for $E$.
  \begin{itemize}
  \item Define $f \colon V_0 \to V_1$ by $x_i\mapsto y_i$ for $1 \le i \le n$ and $x_i\mapsto 0$ for $i>n$. 
  \item Define $h_1 \colon V_0 \to E$ by $x_i \mapsto z_{i-n}$ for $n+1 \le i \le n+\dim E$ and $x_i\mapsto 0$ for all other $i$ (possible since $m \ge n + \dim E$).
    
  \item Define $h_2 \colon F \to V_0$ by $w_i \mapsto x_{i+n+\dim E}$ for $1 \le i \le \dim F$ (possible since $m \ge n + \dim E + \dim F$).
  \item Pick $n$ distinct nonzero complex numbers $c_1,\dots,c_n$ and define $g \colon V_1 \to V_0$ by $y_i \mapsto c_i x_i$ for all $i$.
  \end{itemize}
 Proposition~\ref{prop:complex-CI} applies (with $B=V_0$, $A_1=E$, $A_2=F$, and $C_1=C_2=V_1$) by our assumption \eqref{eqn:ineq}, so we see that $Z$ is irreducible. Hence $U$ is dense, and since $\pi$ is a projective morphism, it suffices to show that $U \subseteq \pi(Y)$.

  So suppose that $(f,g,h_1,h_2) \in U$. Pick an ordering of the eigenvalues of $fg$ and for each $i$, let $R_{1,i}$ be the span of the eigenspaces of the first $i$ eigenvalues in this ordering. Next, since $h_1h_2=0$ and $fh_2=0$, we have $h_2(F) \subseteq \ker \psi$. Also, since $\psi$ is surjective, we have
  \[
    \dim \ker \psi = m - n - \dim E \ge p-q,
  \]
  so we may choose a $(p-q)$-dimensional subspace $R'_0\subseteq \ker \psi$ which contains $h_2(F)$. Finally, since $fg$ is invertible, we see that $g(V_1) \cap \ker \psi = 0$. So for all $i$, we set $R_{0,i+p-q} = R'_0 + g(R_{1,i})$, which has dimension $i+p-q$. Since $f(R'_0)=0$ and $fg(R_{1,i}) = R_{1,i}$, we are done.

  (2) First, $\wt{\cZ}$ is reduced by \cite[Proposition 3.1]{superres}. An ordering of the eigenvalues of $fg$ is the same thing as choosing a point in the prime spectrum of the splitting ring of the characteristic polynomial, which implies the factorization of $\rho$.

  (3) This is similar to (2). The difference is that in order to define the subspace $R_0$ and $R_1$ in the preimage of a point in $U$, it suffices to instead choose $q$ of the eigenvalues of $fg$ rather than a complete ordering. This choice is the same as a point in the $(q,n-q)$-factorization ring of the characteristic polynomial of $fg$.
\end{proof}

The fact that $\rho \colon \cY \to Z$ is surjective gives us the following normal form for points in $Z$ and some information about the discriminant $\Delta$.

\begin{corollary} \label{cor:normal-form}
  \begin{enumerate}
  \item   If $(f,g,h_1,h_2) \in Z$, then there is a choice of bases for $V_0,V_1,E,F$ such that
  \begin{align*}
    f= \begin{bmatrix} \smash[b]{\underbrace{0}_{p-q}} & f^\ru & f' \end{bmatrix}, \qquad g = \begin{bmatrix} g' \\ g^\ru \\ 0 \end{bmatrix} \kern-.8em 
    \begin{array}{c} \left.\vphantom{g'} \right\}  {\scriptstyle p-q}
      \\ \vphantom{g^u} \\ \vphantom{0} \end{array},
    \qquad
    h_1 = \begin{bmatrix} \smash[b]{\underbrace{0}_{n+p-q}} & h_1' \end{bmatrix}, \qquad h_2 = \begin{bmatrix} h_2' \\ 0 \end{bmatrix} \kern-.8em 
    \begin{array}{c} \left.\vphantom{h_2'} \right\}  {\scriptstyle p-q}
      \\ \vphantom{0} \end{array},
  \end{align*}
  where $f^\ru$ and $g^\ru$ are upper-triangular $n \times n$ matrices, and the underbraces, respectively right braces, indicate the number of columns, respectively rows, otherwise.

\item With the notation above, $\chi(u)$ is also the characteristic polynomial of $f^\ru g^\ru$, and if the discriminant of $\chi(u)$ is $0$, then we may also choose the bases so that the first two diagonal entries of $f^\ru g^\ru$ are the same.
  
\item The hypersurface in $Z$ defined by $\Delta$, $V(\Delta)$, is irreducible; in fact, any two points can be joined by an irreducible rational curve.
\end{enumerate}
\end{corollary}

\begin{proof}
  (1) This follows directly from the fact that $\rho \colon \cY \to Z$ is surjective.
  
  (2) We can adapt the proof above; let $U$ be the open subset of $V(\Delta)$ in which $\psi \colon V_0 \to E \oplus V_1$ is surjective, the eigenvalues of $fg$ are nonzero and have exactly one repeated value, and $h_2$ is injective. In that case, let $\lambda$ be the repeated root, and in the previous proof, we can choose $R_{1,1}$ to be the span of any eigenvector of eigenvalue $\lambda$ and choose $R_{1,2} = \ker(fg-\lambda)^2$, the whole (generalized) eigenspace for $\lambda$.

  (3) Using (2), the proof of \cite[Proposition 5.10]{superres} can be adapted fairly easily. We note that the conditions $h_1g=fh_2=h_1h_2=0$ are automatic from the forms of the matrices, so they do not have to be accounted for separately.
\end{proof}

\begin{proposition} \label{prop:piY-CI}
  $Z$ is a complete intersection in $\wt{W}$ defined by the equations $h_1g=fh_2=0$ and is an irreducible variety with rational singularities.
\end{proposition}

\begin{proof}
  First consider the equations $h_1g=fh_2=h_1h_2=0$ in $\bC[W^\re]$. This puts us in the context of Proposition~\ref{prop:complex-CI} with $A_1=F$, $A_2 = E$, $B = V_0$, $C_1=C_2=V_1$. In this case, the inequality reduces to $\dim F + \dim E + n \le m$, which holds by our assumption. So the equations are a regular sequence in $\bC[W^\re]$ and define an integral scheme that has rational singularities. A fortiori, the equations $h_1g=fh_2=0$ are a regular sequence in $\wt{S}$.
\end{proof}

\begin{proposition}
  $\wt{\cZ}$ is a normal and Cohen--Macaulay variety.
\end{proposition}

The same is true for $\wt{Z}$, but we will prove something stronger shortly, so we do not address it now.

\begin{proof}
  First, $Z$ is Cohen--Macaulay by Proposition~\ref{prop:piY-CI}. Since $\bC[\wt{\cZ}]$ is a finite rank free $\bC[Z]$-module \cite[Proposition 3.1(a)]{superres}, it is also Cohen--Macaulay.
  
  Next, we use Proposition~\ref{prop:split-normal} to prove that $\wt{\cZ}$ is normal, which has 3 conditions. Again appealing to Proposition~\ref{prop:piY-CI}, we know that $\bC[Z]$ is a normal domain. This immediately implies (1). Since $\Delta$ is not identically 0, as can be seen from the normal form in Corollary~\ref{cor:normal-form}(1), (2) also holds. For (3), since $Z$ is normal, its singular locus has codimension at least 2, so it suffices to show that $V(\Delta, \partial \Delta) \cap Z_{\rm sm}$ (the subscript denotes the smooth locus) has codimension at least 2. If $x \in Z_{\rm sm}$, then $x \in V(\Delta, \partial \Delta)$ if and only if $\Delta(x)=0$ and its differential $d\Delta$ is 0 in the fiber of the cotangent bundle over $x$.

  From Corollary~\ref{cor:normal-form}(3), we know that $V(\Delta)$ is irreducible, so it suffices to show that $V(\Delta,\partial \Delta)$ is a proper subset of $V(\Delta)$. To do this, let $\lambda, \mu_1,\dots, \mu_{n-2}$ be distinct complex numbers, let $A$ be the diagonal square matrix of size $n-2$ whose entries are $\mu_1,\dots,\mu_{n-2}$, and consider the following $\bC[\epsilon]/(\epsilon^2)$-point $\tilde{x}$ of $Z$ using the notation of the normal form in Corollary~\ref{cor:normal-form}:
  \[
    f^\ru = \begin{bmatrix} \lambda & 1 \\ \epsilon & \lambda \\ & & A \end{bmatrix}, \quad f' = 0,\quad g' = 0, \quad g^\ru = {\rm id}_n, \quad h_1' = 0, \quad h'_2 = \begin{bmatrix} {\rm id}_{\dim F} \\ 0 \end{bmatrix},
  \]
  where ${\rm id}_d$ denotes the $d \times d$ identity matrix, and in $h'_2$, the $0$ refers to a row of $p-q-\dim F$ zeroes. Let $x$ be the corresponding $\bC$-point (i.e., setting $\epsilon=0$). First, the map $V_1 \oplus F \to V_0$ given by the sum $g + h_2$ has full rank, so by the proof of Proposition~\ref{prop:complex-CI}, $x \in Z_{\rm sm}$. Also, by inspection, $x \in V(\Delta)$. However, the discriminant of $\chi(u)$ at $\tilde{x}$ is a nonzero scalar multiple of $\epsilon$, and hence $x \notin V(\Delta, \partial \Delta)$.
\end{proof}

Our next goal is to show that $\wt{\cZ}$ and $\wt{Z}$ have rational singularities. Since $\wt{Z}$ is a quotient of $\wt{\cZ}$ by a finite group (Proposition~\ref{prop:fact-ring}), it suffices to deal with $\wt{\cZ}$ \cite[Corollaire]{boutot}. The proof is very similar to the one used in \cite[\S 5.6]{superres}, so we will use similar notation.

Define the following closed subsets of $Z$:
\begin{itemize}
\item $D_0$ is the locus where $h_2 \colon F \to V_0$ is not injective.
\item $D_1$ is the locus such that $\chi(0)=0$.
\item $D_2$ is the locus where $\chi(u)$ has a repeated root.
\item $D_3$ is the locus where at least one of the following happens:
  \begin{itemize}
  \item $\chi(u)$ has a triple root, or
  \item $\chi(u)$ has two repeated roots, or
  \item $\chi(u)$ has a unique repeated root, but the corresponding Jordan block of $fg$ is a scalar matrix.
  \end{itemize}
\item $D_4 = D_1 \cap D_2$.
\item $D_5 = D_0 \cup D_1 \cup D_3$.
\item $D_6 = D_0 \cap D_2$.
\end{itemize}

Finally, we set $U_i = Z\setminus D_i$ for all $i$. Given any map $Z' \to Z$, let $U_i(Z')$ or $D_i(Z')$ denote the inverse image of $U_i$ or $D_i$, respectively, in $Z'$.

\begin{lemma} \label{lem:codim-lemma}
  \begin{enumerate}
  \item Each of $D_3(\cY)$, $D_4(\cY)$, and $D_6(\cY)$ has codimension $\ge 2$ in $\cY$.
  \item If $p-q=\dim F$, then $\rho \colon U_5(\cY) \to U_5(\wt{\cZ})$ is an isomorphism and hence $\rho \colon \cY \to \wt{\cZ}$ is birational.
  \end{enumerate}
\end{lemma}

\begin{proof}
(1) We will adapt the proof of \cite[Proposition 5.12]{superres}.

By equivariance, the restrictions of $D_3(\cY)$, $D_4(\cY)$, or $D_6(\cY)$ to any fiber over $\Fl(p-q,\dots,p-q+n; V_0) \times \Fl(1,\dots,n-1;V_1)$ is isomorphic to any other. So it suffices to show that in a given fiber of the vector bundle, these restrictions have codimension $\ge 2$. Pick bases of $V_0$ and $V_1$ adapted to the particular pair of flags, i.e., for each subspace of dimension $i$, it is the span of the first $i$ basis vectors. We also pick bases for $E$ and $F$. With respect to these bases, we may assume that $(f,g,h_1,h_2)$ are in the form given by Corollary~\ref{cor:normal-form} and we will continue to use the notation there.

With this setup, the proof that $D_3(\cY)$ and $D_4(\cY)$ have codimension $\ge 2$ in $\cY$ now follows essentially in the same way as in \cite[Proposition 5.12]{superres}, so we will just explain $D_6(\cY)$. Note that the variables in $h_2'$ (the first $p-q$ rows of $h_2$) are disjoint from the variables used in the equations that define $D_2(\cY)$. In particular, each minor of $h_2'$ of size $p-q$ is a nonzerodivisor on $D_2(\cY)$ and there is at least one of them, so intersecting with $D_0(\cY)$ strictly decreases the codimension, i.e., the codimension of $D_6(\cY)$ is at least 2.

    (2) This is analogous to the first part of the proof of \cite[Proposition 5.13]{superres} with one additional observation needed: in the proof of Proposition~\ref{prop:piY-eqns} above, the subspace $R_0'$ is uniquely determined since $\dim h_2(F) = p-q$.
\end{proof}

\begin{proposition} \label{prop:fact-ring-vanish}
  \begin{enumerate}
  \item $\wt{\cZ}$ and $\wt{Z}$ have rational singularities.
  \item   $\pi_* \cO_Y = \cO_{\wt{Z}}$ and $\rR^i \pi_* \cO_Y = 0$ for all $i>0$.
  \end{enumerate}
\end{proposition}

\begin{proof}
  (1) First assume that $p-q=\dim F$. In that case, $U_5(\wt{\cZ})$ is smooth by Lemma~\ref{lem:codim-lemma}(b), and hence also has rational singularities.  Also, the map $U_2(\wt{\cZ}) \to U_2$ is \'etale \cite[Proposition 3.1(d)]{superres} (this part does not use that $p-q=\dim F$). Since $Z$ has rational singularities (Proposition~\ref{prop:piY-CI}), so does $U_2$, and so $U_2(\wt{\cZ})$ also has rational singularities. 

  The complement of $U_2(\cY) \cup U_5(\cY)$ in $\cY$ is
  \[
    D_2(\cY) \cap (D_0(\cY) \cup D_1(\cY) \cup D_3(\cY)) = D_6(\cY) \cup D_4(\cY) \cup (D_2(\cY) \cap D_3(\cY)).
  \]
  By Lemma~\ref{lem:codim-lemma}(a), each of these three subsets in the union on the right side has codimension at least 2. So we can use \cite[Proposition 4.1]{superres} to conclude that $\wt{\cZ}$ has rational singularities.

  Now we consider the general case where we only assume that $p-q\ge \dim F$. In that case, define $F' = F \oplus \bC^{p-q-\dim F}$ and let $\wt{\cZ}'$ be the variety constructed in the same way as $\wt{\cZ}$, except with $F'$ replacing $F$. Then we have a natural inclusion $\bC[\wt{\cZ}] \to \bC[\wt{\cZ}']$ which admits a splitting as a $\bC[\wt{\cZ}]$-module. Then $\wt{\cZ}$ inherits the rational singularities property from $\wt{\cZ}'$ by \cite[Th\'eor\`eme]{boutot}.

  As mentioned before, using \cite[Corollaire]{boutot}, we also get that $\wt{Z}$ has rational singularities since it is a quotient of $\wt{\cZ}$ by the finite group $\fS_q \times \fS_{n-q}$ (Proposition~\ref{prop:fact-ring}).

  (2) Since $\wt{\cZ}$ has rational singularities, we will apply \cite[Proposition 4.2]{superres}. We take $U$ to be the open subset defined in the proof of Proposition~\ref{prop:piY-eqns}. From the proof, we see that if $x \in U$, then $\pi^{-1}(x)$ is isomorphic to the Grassmannian of $(p-q-\dim F)$-dimensional subspaces of a vector space of dimension $m-n-\dim E - \dim F$, namely, $\ker \psi / h_2(F)$.
\end{proof}

\begin{corollary} \label{cor:TorwtS}
  Set $A = \Fact^{(q,n-q)}_{\bC}(u^n) \cong \rH^\bullet_{\rm sing}(\Gr_q(V_1); \bC)$.
  \begin{enumerate}
  \item $A \otimes T = \Tor^S_0(\pi_* \cO_Y, \bC)$, and each Tor group $\Tor^S_i(\pi_* \cO_Y, \bC)$ is a free graded $A \otimes T$-module where the module structure comes from the multiplication on Tor.
    
      \item  The degree $a$ component of $\Tor^S_i(\pi_* \cO_Y, \bC)$ is nonzero only if $a \ge i$ and $a-i$ is even. In that case, we have
  \[
    \Tor^S_i(\pi_* \cO_Y, \bC)_{2j+i} = A^{2j} \otimes \bigwedge^i((F \otimes V_1^*) \oplus (V_1 \otimes E^*)) \otimes T.
  \]
\end{enumerate}
\end{corollary}

\begin{proof}
(1)  By Propositions~\ref{prop:piY-eqns} and \ref{prop:piY-CI}, the minimal free resolution of the $\wt{S}$-module $\bC[Z]$ is given by the Koszul complex
  \[
    \wt{S} \otimes \bigwedge^\bullet((F \otimes V_1^*) \oplus (V_1 \otimes E^*)).
  \]
  By Proposition~\ref{prop:fact-ring-vanish}, we have an isomorphism of $\wt{S}$-modules
  \[
    \pi_* \cO_Y \cong \Fact^{(q,n-q)}_{\bC[Z]}(\chi).
  \]
  Hence $\pi_*\cO_Y$ is a free $\bC[Z]$-module and its space of minimal generators is identified with $\Fact^{(q,n-q)}_{\bC[Z]}(\chi) \otimes_{\wt{S}} \bC \cong \Fact^{(q,n-q)}_\bC(u^n)$.
  In particular, the minimal free resolution of the $\wt{S}$-module $\pi_*\cO_Y$ is a direct sum of Koszul complexes, which we can write as
  \[
    \wt{S} \otimes A \otimes \bigwedge^\bullet((F \otimes V_1^*) \oplus (V_1 \otimes E^*)).
  \]
  All of the nonzero differentials in this resolution are of bidegree $(1,1)$. Hence they all become 0 when we apply $-\otimes_S \bC$. Since $\tilde{S} \otimes_S \bC = T$, this gives isomorphisms of $T$-modules
  \[
    A \otimes T \otimes \Tor_i^{\wt{S}}(\bC[Z], \bC) \to A \otimes \Tor_i^S(\bC[Z], \bC)
  \]
  which is compatible with the multiplication on Tor and the identification $\Tor_0^S(\bC[Z],\bC) = T$.
  Finally, we have an isomorphism of $\wt{S}$-modules
  \[
    A \otimes \bC[Z] \to \pi_*\cO_Y
  \]
  which gives an isomorphism for all $i\ge 0$ on Tor groups:
  \[
    A \otimes \Tor^{\wt{S}}_i(\bC[Z], \bC) \to \Tor^{\wt{S}}_i(\pi_*\cO_Y, \bC).
  \]
  In particular, each Tor group $\Tor^{\wt{S}}_i(\pi_*\cO_Y, \bC)$ is a free $A$-module and the module structure comes from the natural multiplication on Tor. Combining this with the $T$-module isomorphism above finishes the proof of (1).

  (2) In the proof above, we showed that
  \[
    \Tor_i^S(\pi_* \cO_Y, \bC) = A \otimes \bigwedge^i( (F \otimes V_1^*) \oplus (V_1 \otimes E^*) ) \otimes T,
  \]
  so it remains to determine the graded pieces. Considering the bigrading on $\wt{S}$, $A^{2k}$ has bidegree $(2k,0)$ since the entries of $fg$ all have bidegree $(2,0)$, and this factorization ring is constructed from the characteristic polynomial of $fg$ (see the comments about grading in \S\ref{ss:fact}). The spaces $F \otimes V_1^*$ and $V_1 \otimes E^*$ have bidegree $(1,1)$, and $T$ is a sum of spaces of bidegrees $(0,i)$. When restricting to the $\bZ$-grading on $S$, we forget the second factor of the bidegree, which gives the claimed result.
\end{proof}

\subsection{Proof part 3: algebra upstairs} \label{ss:alg-up}

We can give alternative descriptions of $Y''$, $Y'$, and $Y$ in terms of vector bundles over $X_\bos$. Introduce the following vector bundles over $X_\bos$:
\begin{align*}
  \eta'' &= (F \otimes \cR_0^*) \oplus (\cQ_0 \otimes E^*),\\
  \xi' &= (\cR_0 \otimes \cQ_1^*) \oplus (\cR_1 \otimes \cQ_0^*),\\
  \eta' &= W^* / \xi',\\
  \xi &= (\cQ_1 \otimes E^*) \oplus (F \otimes \cR_1^*).
\end{align*}

\begin{proposition} \label{prop:zero-sections}
  \begin{enumerate}
  \item $Y'' = \Spec_{X_\bos}(\Sym(W^* \oplus \eta''))$ (the subscript denotes taking the relative spectrum over $X_\bos$).
  \item $Y' = \Spec_{X_\bos}(\Sym(\eta' \oplus \eta''))$. In particular, we have an exact sequence (Koszul complex)
    \[
      \cdots \to \bigwedge^k(\xi') \otimes \cO_{Y''} \to \cdots \to \xi' \otimes \cO_{Y''} \to \cO_{Y''} \to \cO_{Y'} \to 0.
    \]
  \item $Y$ is the zero locus of a regular section of $\xi$ over $Y'$. In particular, $Y$ is locally a complete intersection and we have an exact sequence (Koszul complex)
    \[
      \cdots \to \bigwedge^k \xi \otimes \cO_{Y'} \to \cdots \to \xi \otimes \cO_{Y'} \to \cO_{Y'} \to \cO_{Y} \to 0.
    \]
  Furthermore, $Y$ is irreducible and has rational singularities.
  \end{enumerate}
\end{proposition}

\begin{proof}
  (1) From the definition, $Y''$ is the total space of the vector bundle $W \times \hom(\cQ_0, E) \times \hom(F, \cR_0)$ over $X_\bos$.

  (2) The conditions $f(R_0) \subseteq R_1$ and $g(R_1) \subseteq R_0$ are locally linear conditions on $Y''$ which are given by the vanishing of a section of $\hom(\cR_0,\cQ_1) \times \hom(\cR_1,\cQ_0) = (\xi')^*$.
  
  (3) Over $Y'$, we have $h_1(\cR_0) = 0$ and $g(\cR_1) \subseteq \cR_0$, and hence have a composition $\cQ_1 \xrightarrow{g} \cQ_0 \xrightarrow{h_1} E$. Similarly, we also get a composition $F \xrightarrow{h_2} \cR_0 \xrightarrow{f} \cR_1$. By definition, $Y$ is the reduced scheme of the zero locus of this pair of sections. Note that \eqref{eqn:ineq} can be rewritten as
  \[
    \dim F + \rank \cR_1 \le \rank \cR_0, \qquad \rank \cQ_1 + \dim E \le \dim \cQ_0.
  \]
  Finally, we wish to use Corollary~\ref{cor:complex-CI}, but it does not apply literally. To do so, it suffices to pick an affine chart where all vector bundles involved are trivialized. Then we can conclude that this zero section is regular, is already reduced, and that $Y$ is irreducible and has rational singularities (these are all local properties). 
\end{proof}

\begin{proposition} \label{prop:tor-cohom}
  We have
  \[
    \Tor^S_i(\pi_* \cO_Y, \bC) \cong \bigoplus_j \rH^j(X_{\bos}; \bigwedge^{i+j}(\xi \oplus \xi') \otimes \Sym(\eta'')).
  \]
  Furthermore, this isomorphism takes the multiplication on Tor on the left side to the induced multiplication on the right side coming from the exterior algebra on $\xi \oplus \xi'$.
\end{proposition}

\begin{proof}
  Consider the cartesian diagram
  \[
    \xymatrix{ \Spec_{X_\bos}(\Sym(\eta'')) \ar[r]^-{\psi'} \ar[d]_-{\phi'} & Y'' \ar[d]^-\phi \\
      \Spec(\bC) \ar[r]^\psi & W }
  \]
  where $\phi$ is the projection map and $\psi$ is inclusion into the origin, i.e., dual to the map of rings $S \to S/(W^*)$. Note that $\cO_{Y''}$ is free over $\cO_W$, and so $Y''$ and $\Spec(\bC)$ are Tor independent. So we can apply cohomology and base change \stacks{08IB}, which gives an isomorphism
  \[
    \rL \psi^*( \rR \phi_* \cO_Y) = \rR \phi'_*( \rL (\psi')^* \cO_Y).
  \]
  The functors on both sides are monoidal, so this isomorphism will also be compatible with the algebra structures on both sides, which addresses the second part of the proposition once the first part has been established.
  
  Now we take $i$th homology of both sides. The left side gives $\Tor_i^S(\rR \phi_*\cO_Y, \bC)$. To compute the right side, we first use Proposition~\ref{prop:zero-sections} to replace $\cO_Y$ with the Koszul complex $\cO_{Y'} \otimes \bigwedge^\bullet \xi$. The differentials of $\rL(\psi')^*( \cO_{Y'} \otimes \bigwedge^\bullet \xi)$ all vanish since the section of $\xi^*$ used to define this Koszul complex is in the ideal of $\cO_{Y'}$ generated by $\eta'$. Hence we have a direct sum 
  \[
    \rL(\psi')^* \cO_Y \cong \bigoplus_k \rL(\psi')^* (\cO_{Y'} \otimes \bigwedge^k \xi)[-k].
  \]
  To describe the terms in the sum, we again use Proposition~\ref{prop:zero-sections} to replace each $\cO_{Y'} \otimes \bigwedge^k \xi$ by $\cO_{Y''} \otimes \bigwedge^k \xi \otimes \bigwedge^\bullet \xi'$. Now we do not need to derive $(\psi')^*$, and again, the differentials are 0 after applying $(\psi')^*$ since the section of $(\xi')^*$ used to define this Koszul complex is in the ideal of $\cO_{Y''}$ generated by $W^*$.
  Hence the $i$th homology of the right side is
  \[
    \bigoplus_{k,j} \rH^{k+j-i}(\Spec_{X_\bos}(\Sym(\eta'')); \bigwedge^k \xi \otimes \bigwedge^j \xi') = \bigoplus_{k,j} \rH^{k+j-i}(X_\bos; \Sym(\eta'') \otimes \bigwedge^k \xi \otimes \bigwedge^j \xi').
  \]
  Finally, $\rR \pi_*\cO_Y = \rR \phi_* \cO_Y$ as $S$-modules, and Proposition~\ref{prop:fact-ring-vanish} shows that the higher direct images of $\pi_* \cO_Y$ vanish, so we are done.
\end{proof}

\subsection{Finishing the proof}

We choose vector spaces $E$ and $F$ so that $\dim E = \ell(\alpha)$ and $\dim F = \ell(\beta)$ and continue to use the above notation.

By Corollary~\ref{cor:TorwtS}, we get
\[
  \Tor^S_i(\pi_* \cO_Y, \bC)_{2j+i} = \rH^{2j}_{\rm sing}(\Gr_q(V_1); \bC) \otimes \bigwedge^i((F \otimes V_1^*) \oplus (V_1 \otimes E^*)) \otimes T.
\]
But Proposition~\ref{prop:tor-cohom} also gives
\begin{align*}
  \Tor^S_i(\pi_* \cO_Y, \bC)_{2j+i} &=  \rH^{2j}(X_{\bos}; \bigwedge^{2j+i}(\xi \oplus \xi') \otimes \Sym(\eta'')).
\end{align*}

Now combine these two isomorphisms, and sum over $i$ to get an isomorphism:
\begin{align} \label{eqn:compare1}
  \rH^{2j}_{\rm sing}(\Gr_q(V_1); \bC) \otimes \bigwedge^\bullet((F \otimes V_1^*) \oplus (V_1 \otimes E^*)) \otimes T  = \rH^{2j}(X_{\bos}; \bigwedge^{\bullet}(\xi \oplus \xi') \otimes \Sym(\eta'')),
\end{align}
which is an algebra isomorphism if we sum over all $j$.

Taking $\GL(E) \times \GL(F)$-invariants on the left side gives $\rH^\bullet_{\rm sing}(\Gr_q(V_1); \bC)$, while taking $\GL(E) \times \GL(F)$-invariants on the right side gives $\rH^\bullet(X_\bos; \bigwedge^\bullet \xi')$. Hence by Corollary~\ref{cor:TorwtS}, the right side is a free module over $\rH^\bullet(X_\bos; \bigwedge^\bullet \xi')$. More specifically, by equivariance, each $\GL(E) \times \GL(F)$-isotypic component of the right side is a submodule which is also free.

Finally, we take the $\bS_\alpha(E^*) \otimes \bS_\beta(F)$-isotypic component of the above isomorphism with respect to the action of $\GL(E)\times \GL(F)$. We explain this in steps. First, by the Cauchy identity, we have
\begin{align*}
  \bigwedge^\bullet \xi \otimes \Sym(\eta'') &= \bigwedge^\bullet(\cQ_1 \otimes E^*) \otimes \bigwedge^\bullet(F \otimes \cR_1^*) \otimes \Sym(F \otimes \cR_0^*) \otimes \Sym(\cQ_0 \otimes E^*)\\
                                             &= \Sym( (\cQ_0|\cQ_1) \otimes E^*) \otimes \Sym(F \otimes (\cR_0^*|\cR_1^*))\\
  &= \left(\bigoplus_\alpha \bS_\alpha(\cQ_0|\cQ_1) \otimes \bS_\alpha(E^*) \right) \otimes \left( \bigoplus_\beta \bS_\beta(F) \otimes \bS_\beta(\cR_0^*|\cR_1^*) \right).
\end{align*}
In particular, the right hand side of \eqref{eqn:compare1} can be written as
\[
  \bigoplus_{\alpha, \beta} \rH^{2j}(X_{\bos}; \bigwedge^\bullet\xi' \otimes  \bS_\alpha(\cQ_0|\cQ_1) \otimes \bS_\beta(\cR_0^*|\cR_1^*) ) \otimes \bS_\alpha(E^*) \otimes \bS_\beta(F).
\]
On the other hand, by \eqref{eqn:T-decomp} and Proposition~\ref{prop:rational-schur-coord}, we have
\begin{align*}
  \bigwedge^\bullet((F \otimes V_1^*) \oplus (V_1 \otimes E^*)) \otimes T &= \bigwedge^\bullet((F \otimes V_1^*) \oplus (V_1 \otimes E^*)) \otimes \bigoplus_{\lambda, \mu} \bS_\lambda (E^*) \otimes \bS_{[\lambda; \mu]}(V_0) \otimes \bS_\mu F\\
  &= \bigoplus_{\alpha, \beta} \bS_\alpha (E^*) \otimes \bS_{[\alpha; \beta]}(V) \otimes \bS_\beta F.
\end{align*}
Comparing these two decompositions gives a $\GL(V_0) \times \GL(V_1)$-equivariant isomorphism
  \begin{align*}
  \rH^{2j}_{\rm sing}(\Gr_q(V_1); \bC) \otimes \bS_{[\alpha;\beta]}(V)  = \rH^{2j}(X_{\bos}; \bigwedge^{\bullet}\xi' \otimes \bS_\alpha(\cQ_0|\cQ_1) \otimes \bS_\beta(\cR_0^*|\cR_1^*)).
  \end{align*}
  As explained above, if we sum over all $j$, then the above space is a free module over $\rH^\bullet(X_\bos; \bigwedge^\bullet \xi')$.

  Finally, with respect to the $\cJ$-adic filtration on $\bS_\alpha(\cQ) \otimes \bS_\beta(\cR^*)$, we have 
  \[
    \gr(\bS_\alpha\cQ \otimes \bS_\beta(\cR^*)) \cong \bS_\alpha(\cQ_0|\cQ_1) \otimes \bS_\beta(\cR_0^*|\cR_1^*) \otimes \bigwedge^\bullet \xi',
  \]
  and hence we have an $\rE_1$ spectral sequence
  \[
    \rE^{p,q}_1 = \rH^{p+q}(X_{\bos}; \bS_\alpha(\cQ_0|\cQ_1) \otimes \bS_\beta(\cR_0^*|\cR_1^*) \otimes \bigwedge^p \xi') \Longrightarrow \rH^{p+q}(X; \bS_\alpha\cQ \otimes \bS_\beta (\cR^*)).
  \]
  Since all cohomology is concentrated in even degree, this spectral sequence is degenerate, so we conclude that we have a $\GL(V_0) \times \GL(V_1)$-equivariant isomorphism
  \begin{align*}
    \rH^{2j}(X; \bS_\alpha\cQ \otimes \bS_\beta (\cR^*)) &= \rH^{2j}(X_{\bos}; \bS_\alpha(\cQ_0|\cQ_1) \otimes \bS_\beta(\cR_0^*|\cR_1^*) \otimes \bigwedge^\bullet \xi') \\
    &= \rH^{2j}_{\rm sing}(\Gr_q(V_1); \bC) \otimes \bS_{[\alpha;\beta]}(V).
  \end{align*}
From the spectral sequence, the $\rH^\bullet(X;\cO_X)$-module $\rH^\bullet(X; \bS_\alpha \cQ \otimes \bS_\beta (\cR^*))$ has a filtration, and the $\rH^\bullet(X;\cO_X)$-module structure on the associated graded space coincides with the action of $\rH^\bullet(X_\bos; \bigwedge^\bullet \xi')$ by the above discussion. Furthermore, the associated graded module is free and generated in degree 0, and hence the $\rH^\bullet(X;\cO_X)$-module $\rH^\bullet(X; \bS_\alpha \cQ \otimes \bS_\beta(\cR^*))$ is also free and generated in degree $0$.
  
Next, from \cite[Theorem 1.2(a)]{superres}, we have $\rH^\bullet(X; \cO_X) \cong \rH^\bullet_{\rm sing}(\Gr_p(V_1); \bC)$ which is a trivial representation of $\GL(V_0) \times \GL(V_1)$. By Theorem~\ref{thm:schur-irred}, $\bS_{[\alpha;\beta]}(V)$ is an irreducible $\GL(V)$-representation. Since the characters of the irreducible representations are linearly independent, we conclude that the isomorphism above must also be $\GL(V)$-equivariant.

Finally, we prove the statement that the canonical surjection \eqref{eqn:schur-surj}
\[
  (\bS_\alpha V \otimes \bS_\beta (V^*)) \times X \to \bS_\alpha \cQ \otimes \bS_\beta(\cR^*) \to 0
\]
induces surjections on cohomology. Taking the sum over all cohomology, both sides are free $\rH^*(X;\cO_X)$-modules generated in degree $0$, so it suffices to show that we get a surjection on sections. Also, the map is $\GL(V)$-equivariant, so by irreducibility of $\bS_{[\alpha;\beta]}(V)$, it actually suffices to show that the map is nonzero on sections. For that, we consider the reduction modulo $\cJ$, and then we can use Corollary~\ref{cor:rat-schur-surj}.

  \section{Generalizations} \label{sec:misc}

  First, everything can be extended to the setting where $V$ is replaced by a superbundle $\cV$ on a superscheme $X$. Then we can instead use the relative super Grassmannian $\Gr(p|q,\cV)$ (see \cite[Definition 2.18]{FAS} for details and references on this construction), which is a superscheme equipped with a structure map $\pi \colon \Gr(p|q,\cV) \to X$ whose fibers are isomorphic to $\Gr(p|q,V)$. This is again equipped with a tautological short exact sequence
  \[
    0 \to \cR \to \pi^* \cV \to \cQ \to 0.
  \]
  As before, for any partitions $\alpha, \beta$, we have a canonical surjection
  \begin{align} \label{eqn:schur-surj2}
    \pi^*( \bS_\alpha \cV \otimes  \bS_\beta(\cV^*)) \to \bS_\alpha \cQ \otimes \bS_\beta(\cR^*) \to 0.
  \end{align}
  
  We note that the definition of $\bS_{[\alpha;\beta]}(\cV)$ continues to make sense and we get the following generalization:

  \begin{theorem} \label{thm:relative}
    Let $\cV$ be a super vector bundle of rank $m|n$ over a superscheme $X$.
  Let $\alpha, \beta$ be partitions and assume that one of the following conditions holds:
  \begin{enumerate}
  \item $m-n-\ell(\alpha) \ge p -q \ge \ell(\beta)$, or
    
  \item $n-m -\alpha_1\ge  q-p \ge \beta_1$.
  \end{enumerate}

  Let $\cA = \bigoplus_k \rR^k \pi_*(\cO_{\Gr(p|q,\cV)})$. The surjection \eqref{eqn:schur-surj2} induces a surjection on $\rR^k \pi_*$ for each $k$, and induces an isomorphism of graded $\cA$-modules
  \[
    \bigoplus_k \rR^k \pi_* ( \bS_\alpha \cQ \otimes \bS_\beta (\cR^*)) \cong \cA \otimes_{\cO_X} \bS_{[\alpha; \beta]}(\cV),
  \]
  where the right side is a free $\cA$-module generated by $\bS_{[\alpha;\beta]}(\cV)$, and the grading induced by $\cA$ coincides with the cohomological grading on the left side. The isomorphism is natural with respect to automorphisms of $\cV$.
\end{theorem}

\begin{proof}
  From \eqref{eqn:schur-surj2} and the projection formula, we get maps
  \addtocounter{equation}{-1}
  \begin{subequations}
\begin{align} \label{eqn:gr-eqn1}
    \cA^k \otimes \bS_\alpha \cV \otimes \bS_\beta (\cV^*) \to \rR^k \pi_* (\bS_\alpha \cQ \otimes \bS_\beta(\cR^*)).
\end{align}
Furthermore, we have a map $\bS_{\alpha/1} \cV \otimes \bS_{\beta/1} (\cV^*) \to \bS_\alpha \cV \otimes \bS_\beta (\cV^*)$ and hence a composition
\begin{align} \label{eqn:gr-eqn2}
    \cA^k \otimes \bS_{\alpha/1} \cV \otimes \bS_{\beta/1}(\cV^*) \to \rR^k \pi_* (\bS_\alpha \cQ \otimes \bS_\beta(\cR^*)).
\end{align}
Then \eqref{eqn:gr-eqn1} is surjective and \eqref{eqn:gr-eqn2} is 0 since we can check these properties at each stalk over $X$, and then it reduces to Theorem~\ref{thm:main-Q} (while everything is done over $\bC$, we can extend the results to an affine scheme over $\bC$ by a base change argument).

Thus, we get an induced map
\[
  \cA^k \otimes \bS_{[\alpha;\beta]} \cV \to \rR^k \pi_* (\bS_\alpha \cQ \otimes \bS_\beta(\cR^*))
\]
which is an isomorphism (again by checking at stalks). The compatibility with the $\cA$-module structure can be checked in a similar way by reducing to checking at stalks.
\end{subequations}
\end{proof}

This allows us to extend the main theorem to super partial flag varieties.

Let $V$ be a super vector space of dimension $m|n$. Let $p_1|q_1 <\cdots < p_r|q_r <m|n$ be an increasing sequence of superdimensions, let $\bF = \Fl(\bp|\bq, V)$ be the corresponding super partial flag variety, and let $\cR_1 \subset \cR_2 \subset \cdots \subset \cR_r$ be the tautological flag of subbundles of $V\times \bF$. Also let $\cQ_i = V/\cR_i$ be the corresponding quotient.

Before stating the generalization of Theorem~\ref{thm:main-Q}, we will compute the ring structure of $\rH^\bullet(\bF; \cO_\bF)$ in the relevant cases.

\begin{theorem} \label{thm:partialflag}
  \begin{enumerate}
  \item If $m - n \ge p_r - q_r \ge \cdots \ge p_2 - q_2 \ge p_1 - q_1 \ge 0$, then  we have an isomorphism of graded algebras
    \[
      \rH^\bullet(\bF; \cO_\bF) \cong \rH^\bullet_{\rm sing}(\Fl(q_1, q_2,\dots, q_r, \bC^n); \bC).
    \]
  \item If $n - m \ge q_r - p_r \ge \cdots \ge q_2 - p_2 \ge q_1 - p_1 \ge 0$, then     we have an isomorphism of graded algebras
  \[
    \rH^\bullet(\bF; \cO_\bF) \cong \rH^\bullet_{\rm sing}(\Fl(p_1, p_2,\dots, p_r, \bC^m); \bC).
  \]    
\end{enumerate}
In both cases, $\cO_\bF$ is $\cJ$-formal.
\end{theorem}

\begin{proof}
  Since $\Fl(\bp|\bq, V) \cong \Fl(\bq|\bp, V[1])$, (2) reduces to (1), so we will just prove case (1).
  
  Let $\cJ$ be the ideal sheaf generated by the odd degree elements in $\cO_\bF$. From \S\ref{ss:supergeom}, $\cJ/\cJ^2$ is a subbundle of $\fg_1 \times \bF_\bos$; let $Y$ be the subvariety of $\fg_1^* \times \bF_\bos$ cut out by these equations. Also from \S\ref{ss:supergeom}, $Y$ is the space of tuples $(R_{0,\bullet}, R_{1,\bullet}, f,g)$ where
  \begin{itemize}
\item $R_{0,\bullet} \in {\bf Fl}(\bp,V_0)$, 
\item $R_{1,\bullet} \in {\bf Fl}(\bq, V_1)$, 
\item $f \colon V_0 \to V_1$ is a linear map such that $f(R_{0,i}) \subseteq R_{1,i}$ for all $i$, and
\item $g \colon V_1 \to V_0$ is a linear map such that $f(R_{1,i}) \subseteq R_{0,i}$ for all $i$.
\end{itemize}
Let $\chi(u) \in \bC[\fg_1^*][u]$ be the characteristic polynomial of $fg$ and let $\bd = (q_1, q_2-q_1,\dots,q_r-q_{r-1},n-q_r)$. Then the map $\pi \colon Y \to \fg_1^*$ factors through $\wt{Z} = \Spec(\Fact^\bd_{\bC[\fg_1^*]}(\chi))$. To see this, given $(R_{0,\bullet}, R_{1,\bullet}, f,g) \in Y$, then for each $i=1,\dots,r+1$, we have an induced map $fg \colon R_{1,i}/R_{1,i-1} \to R_{1,i}/R_{1,i-1}$ (with the conventions $R_{1,r+1}=V_1$ and $R_{1,0}=0$). If we let $\chi_i$ denote its characteristic polynomial, then we have an ordered factorization $\chi = \chi_1 \cdots \chi_{r+1}$ and $\deg\chi_i = q_i-q_{i-1}$ (with the conventions $q_{r+1}=n$ and $q_0=0$). We note that $\wt{Z}$ is reduced by \cite[Proposition 3.1]{superres} and Proposition~\ref{prop:fact-ring}.

Let $\phi \colon Y \to \wt{Z}$ be the corresponding map. Let $U \subset \wt{Z}$ be the nonempty open subset where $g$ is injective, and the eigenvalues of $fg$ are nonzero and distinct. We claim that $\phi$ is surjective and that for $x \in U$, we have
\[
  \phi^{-1}(x) \cong \prod_{i=0}^{r-1} \Gr((p_{i+1}-p_i) - (q_{i+1}-q_i), \bC^{(m-p_i)-(n-q_i)}).
\]

Since $\phi$ is projective and $U$ is dense, it suffices to show that $U \subseteq \phi(Y)$.  Pick $x \in U$, so that we have corresponding linear maps $f,g$ and an ordered factorization $\chi=\chi_1\cdots\chi_{r+1}$. For each $i=0,\dots,r$, let $R_{1,i}$ be the span of the eigenspaces corresponding to the eigenvalues which are the roots of $\chi_1\cdots \chi_i$. Since $fg$ has distinct eigenvalues, we have $\dim R_{1,i} = q_i$.

We will show by induction on $i$ that we can choose subspaces $R_{0,0} \subseteq \cdots \subseteq R_{0,i} \subseteq V_0$ such that $\dim R_{0,j} = p_j$, $g(R_{1,j}) \subseteq R_{0,j}$, and $f(R_{0,j}) \subseteq R_{1,j}$ for all $j=1,\dots,i$.
For the base case $i=0$, we take $R_{0,0}=0$ and $R_{1,0}=0$.

Assuming the first $i$ spaces have been constructed, consider the induced map $f' \colon V_0/R_{0,i} \to V_1/R_{1,i}$. Then
\[
  \dim \ker f' = m-p_i - (n-q_i) \ge (p_{i+1}-p_i)-(q_{i+1}-q_i)
\]
so we can choose a subspace $R' \subseteq \ker f'$ of dimension $(p_{i+1}-p_i)-(q_{i+1}-q_i)$, and let $R''$ be its preimage in $V_0$. We set $R_{0,i+1} = R'' + g(R_{1,i+1})$. By our assumption on $U$, $\dim R_{0,i+1} = p_{i+1}$, and by construction, $f(R_{0,i+1}) \subseteq R_{1,i+1}$. Also, this exhausts all valid choices for $R_{0,i+1}$.

Thus, we have constructed a point in $\pi^{-1}(x)$, so $\pi$ is surjective and the claim about the fibers follows from our construction. 

Finally, $\wt{Z}$ is a quotient of $\Spec(\Split_{\bC[\fg_1^*]}(\chi) )$ by a product of symmetric groups (Proposition~\ref{prop:fact-ring}). The latter has rational singularities (either by using Proposition~\ref{prop:fact-ring-vanish} with $\dim E=\dim F =0$ or by \cite[Proposition 5.15]{superres}), and so $\wt{Z}$ also has rational singularities \cite[Th\'eor\`eme]{boutot}. By \cite[Proposition 4.2]{superres} and our claim above, we conclude that $\rR^i \phi_* \cO_Y = 0$ for $i>0$ and $\phi_* \cO_Y = \cO_{\wt{Z}}$.

Let $A = \rH^\bullet_{\rm sing}(\Fl(q_1,\dots,q_r, \bC^n), \bC)$. Then we have an isomorphism of graded rings $\phi_* \cO_Y \otimes_{\bC[\fg_1^*]} \bC \cong A$ by Proposition~\ref{prop:fact-flag}. By cohomology and base change (see also \cite[Theorem 5.1.2]{weyman}), we conclude that
\[
  \rH^i(\bF_\bos; \bigwedge^j(\cJ/\cJ^2)) = \begin{cases} 0 & \text{if $i\ne j$}\\  A^i & \text{if $i=j$} \end{cases}.
\]
Since $A^i=0$ if $i$ is odd, the spectral sequence
\[
  \rH^p(\bF_\bos; \bigwedge^{p+q}(\cJ/\cJ^2)) \Longrightarrow \rH^p(\bF; \cO_\bF)
\]
is degenerate and the associated graded ring of the induced filtration on $\rH^\bullet(\bF; \cO_\bF)$ 
coincides with $\rH^\bullet(\bF; \cO_\bF)$ (see also \cite[Proposition 6.13]{superres}), so we have an isomorphism of graded rings $A \cong \rH^\bullet(\bF; \cO_\bF)$.
\end{proof}

\begin{theorem}
 Let $\alpha$, $\beta$ be partitions and assume that one of the following conditions holds:
    \begin{enumerate}
    \item $m - n - \ell(\alpha) \ge p_r - q_r \ge \cdots \ge p_2 - q_2 \ge p_1 - q_1 \ge \ell(\beta)$, or
    \item $n - m - \alpha_1 \ge q_r - p_r \ge \cdots \ge q_2 - p_2 \ge q_1 - p_1 \ge \beta_1$.
    \end{enumerate}
    Then we have a $\GL(V)$-equivariant isomorphism of graded $\rH^\bullet(\bF; \cO_\bF)$-modules
  \[
    \rH^\bullet(\bF; \bS_\alpha (\cQ_r) \otimes \bS_\beta (\cR_1^*)) \cong \rH^\bullet(\bF; \cO_\bF) \otimes \bS_{[\alpha; \beta]}(V)
  \]
  where $\rH^\bullet(\bF;\cO_\bF)$ is a trivial $\GL(V)$-representation and the right side is the free $\rH^\bullet(\bF; \cO_\bF)$-module generated by $\bS_{[\alpha;\beta]}(V)$ in degree $0$.
\end{theorem}

\begin{remark}
  \begin{enumerate}
  \item As in the Grassmannian case, the isomorphism in the above theorem is induced by the canonical surjection
    \[
      \bS_\alpha V \otimes \bS_\beta(V^*) \times \bF \to \bS_\alpha(\cQ_r) \otimes \bS_\beta(\cR_1^*) \to 0.
    \]
    
  \item If $\cV$ is a superbundle over a superscheme $X$, we can form the relative super partial flag variety $\Fl(\bp|\bq, \cV)$. There is a version of the theorem that handles this as well, but we omit it for simplicity. The proof of Theorem~\ref{thm:relative} can be adapted to get this general case.
    
  \item It is plausible that $\bS_\alpha(\cQ_r) \otimes \bS_\beta(\cR_1^*)$ is $\cJ$-formal like in the super Grassmannian case, but it is not clear to us if this follows from our proof.
    \qedhere
  \end{enumerate}
\end{remark}

\begin{proof}
{\bf  As in the Grassmannian case, case (2) follows from case (1), so we will only deal with case (1).}

  For $i=1,\dots,r$, let $\bF(i) = \Fl(p_1|q_1,\dots,p_i|q_i, V)$ so that $\bF(r) = \bF$ and $\bF(1) = \Gr(p_1|q_1,V)$. Define the following vector bundle on $\bF(i)$:
  \[
    \cE(i) = \bS_\alpha (\cQ_i) \otimes \bS_\beta(\cR_1^*).
  \]
  Also set
  \[
    A(i) = \rH^\bullet_{\rm sing}(\Gr(q_i-q_{i-1}, \bC^{n-q_{i-1}}); \bC)
  \]
  using the convention $q_0=0$.

  We claim that we have a graded $\GL(V)$-equivariant isomorphism 
  \[
    \rH^\bullet(\bF(i); \cE(i)) \cong A(1) \otimes A(2) \otimes \cdots \otimes A(i) \otimes \bS_{[\alpha; \beta]}(V)
  \]
  and that $\rH^\bullet(\bF(i); \cE(i))$ is generated by the subspace $A(1) \otimes \cdots \otimes A(i-1) \otimes \bS_{[\alpha;\beta]}(V)$ as an $A(i)$-module.
  
  We prove this by induction on $i$. The base $i=1$ is the main result from this article (Theorem~\ref{thm:main-Q}).

For the induction step, consider the map $\pi(i) \colon \bF(i+1) \to \bF(i)$ which forgets the last subspace; this map realizes $\bF(i+1)$ as the relative Grassmannian $\Gr(p_{i+1}-p_i|q_{i+1}-q_i, \cQ_i)$. We consider the Leray spectral sequence applied to the spaces
  \[
    \bF(i+1) \xrightarrow{\pi(i)} \bF(i) \to \Spec(\bC)
  \]
  which gives the $\rE_2$ spectral sequence
  \[
    \rE_2^{a,b} = \rH^a(\bF(i); \rR^b \pi(i)_* \cE(i+1) ) \Longrightarrow \rH^{a+b}(\bF(i+1); \cE(i+1)).
  \]
  First we simplify the input. Note that $\pi(i)^* \bS_\beta(\cR_1^*) = \bS_\beta(\cR_1^*)$ since pullback commutes with multilinear operations on vector bundles, and the higher derived pullbacks vanish.  Also $\rank (\cQ_i) = m-p_i|n-q_i$ and by our assumption, we have
  \[
    m-p_i - (n-q_i) - \ell(\alpha) \ge (p_{i+1} - p_i) - (q_{i+1}-q_i),
  \]
  so by the projection formula and Theorem~\ref{thm:relative}, we have
  \[
    \rR^b \pi(i)_* \cE(i+1) = A(i+1)^b \otimes \cE(i)
  \]
  and, furthermore, $\rR \pi(i)_* \cE(i+1)$ is a free $A(i+1)$-module. Hence the input to the spectral sequence can be written as
  \[
    A(i+1)^b \otimes \rH^a(\bF(i); \cE(i)).
  \]
  However, $A(i+1)^b=0$ if $b$ is odd, and by induction, $\rH^a(\bF(i); \cE(i))=0$ if $a$ is odd, so the spectral sequence must degenerate. In particular, we have an isomorphism of $A(i+1)$-modules
  \[
    \rH^\bullet(\bF(i+1); \cE(i+1)) \cong A(i+1) \otimes \rH^\bullet(\bF(i); \cE(i)),
  \]
  which proves the claim.

  From the claim, we deduce that $\rH^\bullet(\bF; \cE(r))$ is generated by $\bS_{[\alpha;\beta]}(V)$ as a $\rH^\bullet(\bF;\cO_\bF)$-module, so we get a surjective map
  \[
    \rH^\bullet(\bF;\cO_\bF) \otimes \bS_{[\alpha;\beta]}(V) \to \rH^\bullet(\bF; \cE(r))
  \]
  of $\rH^\bullet(\bF;\cO_\bF)$-modules. Since both sides have the same dimension, it must be an isomorphism, which proves the freeness claim.
\end{proof}
  
\section{Examples and remarks} \label{sec:examples}

We end with a few examples. Example~\ref{ex:ex1} shows how to directly compute the cohomology for one case covered by our main theorem. The next two examples are cases not covered by our main theorem and illustrate how different conclusions can fail outside of our assumptions.

Example~\ref{ex:ex2} gives an example where the $\rH^0$ is as expected, but the total cohomology is not free over $\rH^\bullet(X;\cO_X)$. Furthermore, the sheaf is not $\cJ$-formal, i.e., the spectral sequence coming from the $\cJ$-adic filtration is nondegenerate.

Example~\ref{ex:ex3} gives an example which is $\cJ$-formal, but where $\rH^0$ is strictly larger than the representation we expect, i.e., the canonical map \eqref{eqn:schur-surj} is nonzero, but not surjective, on sections.

  \begin{example} \label{ex:ex1}
    We'll go through a simple special case of our theorem, but illustrate what the computation looks like when done directly.
    
  Consider $m = 3$ and $n = 2$ with $p=q=1$, and we want to compute the cohomology of $\Sym^d \cQ$ (to avoid boundary cases, we will assume $d \ge 2$), whose associated graded sheaf with respect to the $\cJ$-adic filtration is
  \[
    (\Sym^d(\cQ_0) + (\Sym^{d-1}(\cQ_0) \otimes \cQ_1)) \otimes \bigwedge^\bullet (\cR_0 \otimes \cQ_1^*) \otimes \bigwedge^\bullet (\cQ_0^* \otimes \cR_1).
  \]
 In this case, $\rH_{\rm sing}^\bullet(\Gr_1(V_1); \bC) \cong \bC[t]/(t^2)$ where $\deg t = 2$.
  
  We remark here on one source of combinatorial complication. First, $\Sym^d(V)$ has 3 terms as a $\GL(V_0) \times \GL(V_1)$-representation:
  \[
    \Sym^d(V) = \Sym^d(V_0) \oplus (\Sym^{d-1}(V_0) \otimes V_1) \oplus (\Sym^{d-2}(V_0) \otimes \bigwedge^2(V_1)).
  \]
  From the usual Borel--Weil theorem, we have
  \begin{align*}
    \rH^0(X_\bos; \Sym^d(\cQ_0)) &= \Sym^d(V_0),\\
    \rH^0(X_\bos; \Sym^{d-1}(\cQ_0) \otimes \cQ_1) &= \Sym^{d-1}(V_0) \otimes V_1,
  \end{align*}
  so the term $\Sym^{d-2}(V_0) \otimes \bigwedge^2(V_1)$ necessarily requires a contribution from the exterior algebra portion of the sheaf. Navigating these contributions in the general case seems quite difficult to do, which is why we have opted for the indirect approach that we took in this paper.

In general, the associated graded sheaf is a direct sum of the sheaves of the form
\[
  \cR_0^{\otimes a} \otimes \bigwedge^b (\cQ_0^*)   \otimes \Sym^{d-c}(\cQ_0) \otimes \cR_1^{\otimes b} \otimes \cQ_1^{\otimes c-a},
\]
where $a,c \le 1$ and $b\le 2$. Here are all of the terms that give nonzero cohomology $\rH^i$:
  \[
    \begin{array}{c|c|c|c|c}
      a & b & c & i & \rH^i\\      \hline
      0 & 0 & 0 & 0 & \Sym^d(V_0)\\      \hline
      0 & 0 & 1 & 0 & \Sym^{d-1}(V_0) \otimes V_1 \\ \hline
      0 & 1 & 1 & 0 & \Sym^{d-2}(V_0) \otimes \bigwedge^2(V_1)\\ \hline
      1 & 1 & 0 & 2 & \Sym^d(V_0) \\ \hline
      1 & 1 & 1 & 2 & \Sym^{d-1}(V_0) \otimes V_1 \\ \hline
      1 & 2 & 1 & 2 & \Sym^{d-2}(V_0) \otimes \bigwedge^2(V_1)
    \end{array}
  \]
  So we get $\rH^0(X; \Sym^d(\cQ)) = \rH^2(X; \Sym^d(\cQ)) = \Sym^d(V)$, as expected.
\end{example}

\begin{example} \label{ex:ex2}
  Now we slightly modify the parameters in the previous example to get a case that is not covered by our main theorem.

  Consider now $m=n=2$ and $p=q=1$ and again we consider the sheaf $\Sym^d \cQ$ with $d \ge 2$. Again, we have $\rH^\bullet(X;\cO_X) \cong \bC[t]/(t^2)$ with $\deg(t)=2$. Each of $\cR_i,\cQ_i$ on $X_\bos$ is a line bundle, so the associated graded sheaf of $\Sym^d \cQ$ is a direct sum of the sheaves of the form
\[
  \cR_0^{\otimes a} \otimes \cQ_0^{\otimes d-c-b} \otimes \cR_1^{\otimes b} \otimes \cQ_1^{\otimes c-a} 
\]
where $a,b,c \le 1$. Here are all of the terms that give nonzero cohomology $\rH^i$:
  \[
    \begin{array}{c|c|c|c|c}
      a & b & c & i & \rH^i\\      \hline
      0 & 0 & 0 & 0 & \Sym^d(V_0)\\      \hline
      0 & 0 & 1 & 0 & \Sym^{d-1}(V_0) \otimes V_1 \\ \hline
      0 & 1 & 1 & 0 & \Sym^{d-2}(V_0) \otimes \bigwedge^2(V_1)\\ \hline
      1 & 0 & 1 & 0 & \bS_{d-1,1}(V_0) \\ \hline
      1 & 1 & 0 & 1 & \bS_{d-1,1}(V_0)
    \end{array}
  \]
  If the spectral sequence degenerates, then we would have $\rH^1(X;\Sym^d(\cQ)) = \bS_{d-1,1}(V_0)$. However, there is no $\GL(V)$-representation whose underlying $\GL(V_0)\times \GL(V_1)$-representation is $\bS_{d-1,1}(V_0)$, so we conclude that the spectral sequence is nondegenerate and the last two entries of the table cancel each other.

  This leads us to the final calculation $\rH^0(X;\Sym^d \cQ) = \Sym^d V$. This is an irreducible $\GL(V)$-representation (all Schur functors $\bS_\lambda V$ are). However, unlike the previous case, we have $\rH^2(X;\Sym^d \cQ) =0$, so the cohomology is not free over $\rH^0(X;\cO_X)$.
\end{example}

\begin{example} \label{ex:ex3}
  Finally, we illustrate an example where $\rH^0$ is not an irreducible $\GL(V)$-representation (but is indecomposable).
  
  Set $p=q=1$, $m \ge 2$, and $n=1$ and consider the sheaf $\Sym^d(\cR^*)$ on $X=\Gr(1|1,V)$ for $d \ge 2$. In this case, $X_\bos = \Gr(1,V_0)$, $\cJ/\cJ^2 = \cQ_0^* \otimes V_1$, and $\gr^0(\cR^*) = \cR_0^* \oplus V_1^*$. Also, by \cite[Theorem 1.2]{superres}, we have $\rH^0(X;\cO_X) = \bC$ and all higher cohomology vanishes.

  In general, for $a \ge 1$, from the usual Borel--Weil--Bott theorem (for example, see \cite[Corollary 4.1.9]{weyman}), we have
  \[
    \rH^0(X_\bos; \Sym^a \cR_0^* \otimes \bigwedge^b \cQ_0^*) = \bS_{(a,1^b)} (V_0^*).
  \]
  where $1^b$ is a length $b$ sequence consisting of all 1's, and all higher cohomology vanishes.  This gives us
  \begin{align*}
    \rH^0(X_\bos; \gr(\Sym^d\cR^*)) &=
                                      \bigoplus_{b=0}^{m-1} \rH^0(X_\bos; \Sym^d \cR_0^* \otimes \bigwedge^b \cQ_0^* \otimes V_1^{\otimes b})\\
                                    & \qquad \oplus \bigoplus_{b=0}^{m-1} \rH^0(X_\bos; \Sym^{d-1} \cR_0^* \otimes V_1^* \otimes \bigwedge^b \cQ_0^* \otimes V_1^{\otimes b})\\
    &= \bigoplus_{b=0}^{m-1} \left( ( \bS_{(d,1^b)}(V_0^*) \otimes V_1^{\otimes b}) \oplus ( \bS_{(d-1,1^b)}(V_0^*) \otimes V_1^{\otimes b-1} )\right)
  \end{align*}
  and all higher cohomology vanishes. Pick an ordered basis $v_1,\dots,v_{m+1}$ for $V$ where $v_1,\dots,v_m$ are even and $v_{m+1}$ is odd, to get a Borel subalgebra. With respect to this choice, the highest weight appears in $\bS_{d-1}(V_0^*) \otimes V_1^*$ (i.e., the second term with $b=0$). As in \S\ref{ss:irred}, this is the highest weight for $\Sym^d(V^*)$, and so we see that $\Sym^d(V^*)$ is a subrepresentation of $\rH^0(X;\Sym^d(\cR^*))$ as a $\GL(V)$-representation.

  The quotient can be identified as well. First, recall that $\GL(V)$ has a 1-dimensional character, the superdeterminant $\sdet$ (also called the Berezinian); when considering $\fgl(V)$, this becomes the supertrace. When restricted to $\GL(V_0) \times \GL(V_1)$, this is $\det(V_0) \otimes \det(V_1)^{-1}$, or in our case, $\det(V_0) \otimes V_1^*$. If we tensor the quotient by the superdeterminant, then we get a representation whose character agrees with the rational Schur functor $\bS_{[1^{m-2}; d-1]}(V)$ (using Proposition~\ref{prop:rational-schur-coord}). By Theorem~\ref{thm:schur-irred}, $\bS_{[1^{m-2}; d-1]}(V)$ is irreducible, so they are isomorphic as $\GL(V)$-representations.  Hence we have a short exact sequence
  \[
    0 \to \Sym^d(V^*) \to \rH^0(X; \Sym^d(\cR^*)) \to \sdet^{-1} \otimes \bS_{[1^{m-2}; d-1]}(V) \to 0
  \]
  In the notation of \S\ref{ss:irred}, $\Sym^d(V^*)$ has highest weight $(0,\dots,0, -d+1|-1)$ while the quotient has highest weight $(0,\dots,-1,-d+1|0)$. The latter is obtained by subtracting the odd root $e_{m-1}-f_1$ which is orthogonal to $(0,\dots,0,-d+1|-1) + \rho$, so the two weights have the same central character, so it is not clear from the combinatorics if this short exact sequence splits or not.

  However, by \cite[Lemma 2]{gruson}, the dual of the space of sections of an irreducible homogeneous bundle over a super homogeneous space is a highest weight $G$-module, and hence is indecomposable. So in particular, the space of sections is also indecomposable. Thus, the short exact sequence does {\it not} split. We thank an anonymous referee for pointing this out.
\end{example}

We end with a few brief remarks.

\begin{remark}
  While the general problem of computing cohomology of homogeneous bundles over super homogeneous spaces is unlikely to be resolved using the techniques presented here, it is still an interesting problem to identify which homogeneous bundles are $\cJ$-formal. Even in the case of the structure sheaf, this is a nontrivial problem.
\end{remark}

\begin{remark}
  In \cite[Theorem 3.1]{raicu-weyman}, the Tor groups of certain non-reduced thickenings of determinantal varieties are computed, and they are shown to be representations of the general linear Lie superalgebra in \cite{raicu-weyman2}. Notably, the formula for its character shows that it is a sum of terms which are products of something with a $q$-binomial coefficient (also called Gaussian binomial coefficient), which is precisely the form that the Hilbert series of the singular cohomology of the Grassmannian takes (see \eqref{eqn:gr-sing-hilbert}). Hence it is tempting to conjecture that these Tor groups are free over subrings which are isomorphic to such cohomology rings, or even that they are computing the cohomology of some homogeneous bundle.
\end{remark}

\begin{remark}
  In the non-super setting, the general Borel--Weil--Bott theorem shows that every irreducible homogeneous bundle has cohomology in at most one degree and that that degree is 0 if and only if the bundle is indexed by a dominant weight. We might guess that there is a nice class of homogeneous bundles in the super setting which exhibit this behavior; to adjust according to the results of this paper, we would look for a homogeneous bundle $\cE$ over $X$ such that there exists $i>0$ with the property that $\rH^j(X; \cE) \cong \rH^{j-i}(X; \cO_X) \otimes \rH^i(X; \cE)$ for all $j$. We note that the line bundles on projective superspace are one such example (see \cite[Proposition 2.33]{FAS}).
\end{remark}

\end{document}

%% file: preamble.tex
\documentclass[12pt]{amsart}

\usepackage{enumerate, amsmath, amsthm, amsfonts, amssymb, xy,  mathrsfs, graphicx, paralist, fancyvrb}
\usepackage[usenames, dvipsnames]{xcolor}
\usepackage[margin=1in]{geometry} 
\usepackage[bookmarks, colorlinks=true, linkcolor=blue, citecolor=blue, urlcolor=blue]{hyperref}

\input xy
\xyoption{all}

\numberwithin{equation}{section}
\newtheorem{theorem}[equation]{Theorem}
\newtheorem*{thm}{Theorem}
\newtheorem{proposition}[equation]{Proposition}
\newtheorem{lemma}[equation]{Lemma}
\newtheorem{corollary}[equation]{Corollary}

\theoremstyle{definition}
\newtheorem{rmk}[equation]{Remark}
\newenvironment{remark}[1][]{\begin{rmk}[#1] \pushQED{\qed}}{\popQED \end{rmk}}
\newtheorem{eg}[equation]{Example}
\newenvironment{example}[1][]{\begin{eg}[#1] \pushQED{\qed}}{\popQED \end{eg}}
\newtheorem{defn}[equation]{Definition}

\newcommand{\cA}{\mathcal{A}}

\newcommand{\cB}{\mathcal{B}}

\newcommand{\bC}{\mathbf{C}}
\newcommand{\cC}{\mathcal{C}}

\newcommand{\cE}{\mathcal{E}}

\newcommand{\rE}{\mathrm{E}}

\newcommand{\bF}{\mathbf{F}}
\newcommand{\cF}{\mathcal{F}}

\newcommand{\rH}{\mathrm{H}}

\newcommand{\cJ}{\mathcal{J}}

\newcommand{\rL}{\mathrm{L}}

\newcommand{\cM}{\mathcal{M}}

\newcommand{\cO}{\mathcal{O}}

\newcommand{\bP}{\mathbf{P}}

\newcommand{\cQ}{\mathcal{Q}}

\newcommand{\cR}{\mathcal{R}}

\newcommand{\rR}{\mathrm{R}}

\newcommand{\bS}{\mathbf{S}}
\newcommand{\cS}{\mathcal{S}}
\newcommand{\fS}{\mathfrak{S}}

\newcommand{\cV}{\mathcal{V}}

\newcommand{\bX}{\mathbf{X}}

\newcommand{\cY}{\mathcal{Y}}

\newcommand{\bZ}{\mathbf{Z}}
\newcommand{\cZ}{\mathcal{Z}}

\newcommand{\bb}{\mathbf{b}}
\newcommand{\fb}{\mathfrak{b}}

\newcommand{\bd}{\mathbf{d}}

\newcommand{\re}{\mathrm{e}}

\newcommand{\fg}{\mathfrak{g}}

\newcommand{\fh}{\mathfrak{h}}

\newcommand{\bk}{\mathbf{k}}

\newcommand{\bp}{\mathbf{p}}
\newcommand{\fp}{\mathfrak{p}}

\newcommand{\bq}{\mathbf{q}}

\newcommand{\bu}{\mathbf{u}}

\newcommand{\ru}{\mathrm{u}}

\newcommand{\bw}{\mathbf{w}}

\renewcommand{\phi}{\varphi}
\renewcommand{\emptyset}{\varnothing}

\renewcommand{\tilde}[1]{\widetilde{#1}}
\newcommand{\ol}[1]{\overline{#1}}
\newcommand{\ul}[1]{\underline{#1}}
\newcommand{\DS}{\displaystyle}

\newcommand{\arxiv}[1]{\href{http://arxiv.org/abs/#1}{{\tt arXiv:#1}}}

\makeatletter
\def\Ddots{\mathinner{\mkern1mu\raise\p@
\vbox{\kern7\p@\hbox{.}}\mkern2mu
\raise4\p@\hbox{.}\mkern2mu\raise7\p@\hbox{.}\mkern1mu}}
\makeatother

\DeclareMathOperator{\im}{image} 
\DeclareMathOperator{\codim}{codim}

\renewcommand{\hom}{\operatorname{Hom}}

\DeclareMathOperator{\rank}{rank}

\DeclareMathOperator{\Sym}{Sym}

\DeclareMathOperator{\Tor}{Tor}

\DeclareMathOperator{\Spec}{Spec}

\DeclareMathOperator{\gr}{gr}

\newcommand{\GL}{\mathbf{GL}}

\newcommand{\Gr}{\mathbf{Gr}}

\newcommand{\fgl}{\mathfrak{gl}}






%% file: bwfact_v3_arxiv.bbl
\begin{thebibliography}{GHSS}

\bibitem[AW1]{akin-weyman} Kaan Akin, Jerzy Weyman. Minimal free resolutions of determinantal ideals and irreducible representations of the Lie superalgebra $\gl(m \vert n)$. {\it J.\ Algebra} {\bf 197} (1997), no.~2, 559--583.

\bibitem[AW2]{akin-weyman2} Kaan Akin, Jerzy Weyman. The irreducible tensor representations of $\gl(m|1)$ and their generic homology. {\it J.\ Algebra} {\bf 230} (2000), no.~1, 5--23.

\bibitem[AW3]{akin-weyman3} Kaan Akin, Jerzy Weyman. Primary ideals associated to the linear strands of Lascoux's resolution and syzygies of the corresponding irreducible representations of the Lie superalgebra $\gl(m|n)$. {\it J.\ Algebra} {\bf 310} (2007), no.~2, 461--490.

\bibitem[Bo]{boutot} Jean-Fran{\c c}ois Boutot. Singularit\'es rationnelles et quotients par les groups r\'eductifs. {\it Invent. Math.} {\bf 88} (1987), 65--68.
  
\bibitem[BRP]{FAS} Ugo Bruzzo, Daniel Hernandez Ruiperez, Alexander Polishchuk. Notes on fundamental algebraic supergeometry. Hilbert and Picard superschemes. {\it Adv. Math.} {\bf 415} (2023) Art. No. 108890, 115 pp. \arxiv{2008.00700v3}

\bibitem[ChW]{cheng-wang} Shun-Jen Cheng, Weiqiang Wang. {\it Dualities and Representations of Lie Superalgebras}, Graduate Studies in Math. {\bf 144}, American Mathematical Society, 2012.
  
\bibitem[CoW]{comes-wilson} Jonathan Comes, Benjamin Wilson. Deligne's category $\ul{\text{Rep}}(GL_\delta)$ and representations of general linear supergroups. {\it Represent. Theory} {\bf 16} (2012), 568--609. \arxiv{1108.0652v1}
  
\bibitem[Cou]{coulembier} Kevin Coulembier. Bott--Borel--Weil theory, BGG reciprocity and twisting functors for Lie superalgebras. \textit{Transform.\ Groups} \textbf{21} (2016), no.~3, 681--723.
 \arxiv{1404.1416v4}

\bibitem[DCS]{dcs} Corrado De Concini, Elisabetta Strickland. On the variety of complexes. {\it Adv. Math.} {\bf 41} (1981), 57--77.

\bibitem[DS]{DS} M. Duflo, V. Serganova. On associated variety for Lie superalgebras. \arxiv{math/0507198v1}

\bibitem[GHSS]{DS2} Maria Gorelik, Crystal Hoyt, Vera Serganova, Alexander Sherman. The Duflo-Serganova functor, vingt ans apr\`es. \arxiv{2203.00529v2}

  \bibitem[GSS]{GSS} Daniel R. Grayson, Alexandra Seceleanu, Michael E. Stillman. Computations in intersection rings of flag bundles. \arxiv{1205.4190v2}
  
  \bibitem[GS]{gruson} Caroline Gruson, Vera Serganova. Cohomology of generalized supergrassmannians and character formulae for basic classical Lie superalgebras. {\it Proc. Lond. Math. Soc.} {\bf 101} (2010), no.~3, 852--892. \arxiv{0906.0918v2}

\bibitem[H]{huang} Hang Huang. Syzygies of determinantal thickenings. \arxiv{2008.02690v1}
    
  \bibitem[L]{lascoux} Alain Lascoux. Syzygies des vari\'et\'es d\'eterminantales. {\it Adv. in Math.} {\bf 30} (1978), no.~3, 202--237.

  \bibitem[Ma]{manin} Yuri I.~Manin. {\it Gauge Field Theory and Complex Geometry}, second edition. Grundlehren der Mathematischen Wissenschaften {\bf 289}, Springer-Verlag, Berlin, 1997.

\bibitem[MOS]{quot} Alina Marian, Dragos Oprea, Steven V Sam. On the cohomology of tautological bundles over Quot schemes of curves. {\it Algebra Number Theory}, to appear, \arxiv{2211.03901v3}
  
\bibitem[MT]{masuoka} Akira Masuoka, Yuta Takahashi. Geometric construction of quotients $G/H$ in supersymmetry. {\it Transform. Groups} {\bf 26} (2021), 347--375. \arxiv{1808.05753v4}

\bibitem[MZ]{MZ} Akira Masuoka, Alexandr N. Zubkov. Quotient sheaves of algebraic supergroups are superschemes. {\it J. Algebra} {\bf 348} (2011), 135--170. \arxiv{1007.2236v3}
    
\bibitem[Pe]{penkov} I. Penkov. Borel--Weil--Bott theory for classical Lie superalgebras. \textit{J.\ Soviet Math.} \textbf{51} (1990), 2108--2140.

\bibitem[PS]{penkov-serganova} I. Penkov, V. Serganova. Characters of irreducible $G$-modules and cohomology of $G/P$ for the Lie supergroup $G = Q(N)$. \textit{J.\ Math.\ Sci.} \textbf{84} (1997), no.~5, 1382--1412.

\bibitem[PW]{pragacz-weyman} Piotr Pragacz, Jerzy Weyman. Complexes associated with trace and evaluation. Another approach to Lascoux's resolution. {\it Adv.\ in Math.} {\bf 57} (1985), no.~2, 163--207.

\bibitem[RW1]{raicu-weyman} Claudiu Raicu, Jerzy Weyman. The syzygies of some thickenings of determinantal varieties. {\it Proc. Amer. Math. Soc.} {\bf 145} (2017), no.~1, 49--59. \arxiv{1411.0151v2}
  
\bibitem[RW2]{raicu-weyman2} Claudiu Raicu, Jerzy Weyman. Syzygies of determinantal thickenings and representations of the general linear Lie superalgebra. {\it Acta Math. Vietnam.} {\bf 44} (2019), no.~1, 269--284. \arxiv{1808.05649v1}

\bibitem[Sa1]{sam} Steven V Sam. Derived supersymmetries of determinantal varieties. {\it J.\ Commut.\ Algebra} {\bf 6} (2014), no.~2, 261--286. \arxiv{1207.3309v1}

\bibitem[Sa2]{sam-osp} Steven V Sam. Orthosymplectic Lie superalgebras, Koszul duality, and a complete intersection analogue of the Eagon--Northcott complex. {\it J. Eur. Math. Soc (JEMS)} {\bf 18} (2016), no.~12, 2691--2732. \arxiv{1312.2255v2}

\bibitem[Sa3]{bwfact2} Steven V Sam. Borel--Weil factorization for orthosymplectic Grassmannians. In preparation.

\bibitem[SS1]{expos} Steven V Sam, Andrew Snowden. Introduction to twisted commutative algebras. \arxiv{1209.5122v1}
  
\bibitem[SS2]{superres} Steven V Sam, Andrew Snowden. Cohomology of flag supervarieties and resolutions of determinantal ideals. {\it Algebr. Geom.} {\bf 11} (2024), no.~1, 37--70. \arxiv{2108.00504v2}
  
\bibitem[SS3]{superres2} Steven V Sam, Andrew Snowden. Cohomology of flag supervarieties and resolutions of determinantal ideals II. \arxiv{2412.20797v1}

\bibitem[SP]{stacks-project} The Stacks project authors, The Stacks project, \url{https://stacks.math.columbia.edu} (2025)

\bibitem[St]{EC2} Richard P. Stanley. {\it Enumerative Combinatorics Vol. 2}. Cambridge Studies in Advanced Mathematics {\bf 62}, Cambridge University Press, Cambridge, 1999.
  
\bibitem[We]{weyman} Jerzy Weyman. {\it Cohomology of Vector Bundles and Syzygies}. Cambridge Tracts in Math. {\bf 149}, Cambridge University Press, Cambridge, 2003.

\end{thebibliography}
